\documentclass[conference]{IEEEtran}
\ifCLASSINFOpdf
\else
\fi
\usepackage{amsmath,amsfonts,amssymb,amsthm,stmaryrd,mathrsfs,mathdots,tikz,amsfonts,amsbsy,float,cmll,xcolor}
\usepackage[small]{diagrams}
\usepackage{bussproof}
\usetikzlibrary{tikzmark}
\usepackage{bm}
\usepackage{hyperref}

\theoremstyle{definition}
\newtheorem{definition}{Definition}

\theoremstyle{plain}
\newtheorem{theorem}[definition]{Theorem}

\newtheorem{lemma}[definition]{Lemma}
\newtheorem{corollary}[definition]{Corollary}
\theoremstyle{remark}
\newtheorem*{remark}{Remark}
\newtheorem*{notation}{Notation}
\newtheorem*{convention}{Convention}
\newarrow{Corresponds}<--->

\begin{document}
\title{On the Unity of Logic: a Sequential, \\ Unpolarized Approach}

\author{\IEEEauthorblockN{Norihiro Yamada}
\IEEEauthorblockA{Department of Computer Science\\
University of Oxford\\
Wolfson Bulding, Parks Road, Oxford OX1 3QD, United Kingdom\\
Email: norihiro.yamada@cs.ox.ac.uk}}


%Telephone: (800) 555--1212, Fax: (888) 555--1212}


\maketitle

\begin{abstract}
The present work aims to give a \emph{unity} of logic via standard \emph{sequential}, \emph{unpolarized} games.
Specifically, our vision is that there must be mathematically precise concepts of \emph{linear refinement} and \emph{intuitionistic restriction} of logic such that the linear refinement of \emph{classical logic (CL)} coincides with \emph{(classical) linear logic (LL)}, and its intuitionistic restriction with the linear refinement of \emph{intuitionistic logic (IL)} into \emph{intuitionistic LL (ILL)}.
However, LL is, in contradiction to the name, cannot be the linear refinement of CL at least from the \emph{game-semantic} point of view due to its \emph{concurrency} and \emph{polarization}.
In fact, existing game semantics of LL employs concurrency, which is rather exotic to game semantics of ILL, IL or CL.
Also, linear negation in LL brings polarization to logic, which is never true in (game semantics of) ILL, IL or CL.
In search for the linear refinement of CL (or the \emph{classicalization} of ILL), we carve out (a calculus of) \emph{linear logic negative (LL$^-$)} from (the two-sided sequent calculus of) LL by discarding linear negation, restricting the rules Cut, $\otimes$R, $\invamp$L, $\&$R, $\oplus$L and $\multimap$R (for they cause concurrency) in a certain way, and adding \emph{distribution rules} to recover these rules (except $\multimap$R) and give a translation of sequents $\boldsymbol{\Delta \vdash \Gamma}$ for CL into the sequents $\boldsymbol{\oc \Delta \vdash \wn \Gamma}$ for LL$^-$.
We then give a categorical semantics of LL$^-$, for which we introduce \emph{why not monad} $\boldsymbol{\wn}$, dual to the well-known \emph{of course comonad} $\boldsymbol{\oc}$, giving a categorical translation $\boldsymbol{\Delta \Rightarrow \Gamma \stackrel{\mathrm{df. }}{=} \wn (\Delta \multimap \Gamma) \cong \oc \Delta \multimap \wn \Gamma}$ of CL into LL$^-$, which is the Kleisli extension of the standard translation $\boldsymbol{\Delta \Rightarrow \Gamma \stackrel{\mathrm{df. }}{=} \oc \Delta \multimap \Gamma}$ of IL into ILL.
%These translations are applicable to the sequent calculi for the logics as well. 
Moreover, we instantiate the categorical semantics by a \emph{fully complete} (sequential, unpolarized) game semantics of LL$^-$ (without atoms), for which we introduce \emph{linearity} of strategies.
Moreover, employing the above categorical translations, it automatically leads to game semantics of ILL, IL and CL as well.
Thus, we establish a sequential, unpolarized unity of logic, where discarding the co-Kleisli construction $\boldsymbol{(\_)_{\oc}}$ and/or the Kleisli construction $\boldsymbol{(\_)^{\wn}}$, and imposing \emph{well-bracketing} on strategies capture linear refinement and intuitionistic restriction of logic in a syntax-independent manner, respectively.
\end{abstract}


% \fi
\IEEEpeerreviewmaketitle

\section{Introduction}
\subsection{Linear Logic}
\emph{Linear logic (LL)} \cite{girard1987linear} is often said to be \emph{resource-conscious} or \emph{resource-sensitive} because it requires proofs to consume each premise exactly once to produce a conclusion. %in fact, it can be seen informally yet very naturally as the logic of \emph{resource consumption/production} \cite{girard1995linear}.
%By this computational nature, it has been having tremendous impacts not only on mathematical logic but also on theoretical computer science \cite{girard1995advances}.
One of the achievements of LL is: Like \emph{classical logic (CL)} \cite{shapiro2018classical,troelstra2000basic} it has an \emph{involutive} negation, and more generally the \emph{De Morgan dualities} \cite{girard1987linear,girard1995linear}, in the strict sense (i.e., not only up to logical equivalence), called \emph{linear negation} $(\_)^\bot$, while like \emph{intuitionistic logic (IL)} \cite{heyting1930formalen,troelstra1988constructivism,troelstra2000basic} it has \emph{constructivity} in the sense of non-trivial semantics \cite{girard1989proofs}, where note that neither CL nor IL (in the form of the sequent calculi \emph{LK} and \emph{LJ} \cite{gentzen1935untersuchungen}) achieves both the dualities and the contructivity \cite{girard1989proofs,troelstra2000basic}.

Strictly speaking, LL has both classical and intuitionistic variants, \emph{CLL} and \emph{ILL}, and LL usually refers to CLL \cite{abramsky1993computational,mellies2009categorical}.
Let us call the standard (and two-sided) sequent calculi for LL and ILL \cite{girard1987linear,abramsky1993computational,troelstra2000basic,mellies2009categorical} \emph{LLK} and \emph{LLJ}, respectively.

\if0
\subsection{Strict De Morgan Dualities in Linear Logic}
\label{LinearNegation}
Nevertheless, although the strict dualities play a central role in LL, it gives us the following puzzles.
%First, linear negation is not obtained by \emph{linearly refining} LK.
First, recall that LLJ is obtained from LJ by eliminating the \emph{weakening} and the \emph{contraction} rules, where each logical connective generates the \emph{context-sharing} and the \emph{non-context-sharing} variants (see Sect. 9.3 of \cite{troelstra2000basic} for this point); in this sense, ILL is the \emph{linear refinement} of IL.
However, LLK is obtained from LK by the linear refinement \emph{plus introducing the strict dualities}. 
%In other words, CLL is not just the linear refinement of CL due to the exotic linear negation.

Also, the strict dualities have nothing to do with \emph{classicalization} of ILL either.
Recall that LJ is obtained from LK by restricting sequents to the ones that have at most one formula on the RHS, which we call \emph{intuitionistic restriction}, while LLJ is obtained from LLK by the intuitionistic restriction \emph{plus deleting top $\top$, zero $0$, par $\invamp$, why not $\wn$ and the strict dualities}.\footnote{There are other variants on which constants to be excluded from ILL \cite{abramsky1993computational,troelstra2000basic,mellies2009categorical}.} 
%Although the elimination of top, zero, par and why not is essentially the same as the intuitionistic restriction (as we shall see), it is not the case for the elimination of the strict dualities.
\fi

\subsection{Game Semantics}
\emph{Game semantics} \cite{abramsky1999game,hyland1997game} refers to a particular kind of \emph{semantics of logic and computation} \cite{winskel1993formal,gunter1992semantics,amadio1998domains}, in which formulas (or types) and proofs (or programs) are interpreted as \emph{games} and \emph{strategies}, respectively. %\footnote{Strictly speaking, strategies to model proofs are a particular kind of strategies to model programs; specifically, they are \emph{winning} ones.}

A game is a certain kind of a rooted forest whose branches correspond to possible developments or \emph{(valid) positions} of the `game in the usual sense' (such as chess and poker). These branches are finite sequences of \emph{moves} of the game; a play of the game proceeds as its participants, \emph{Player} who represents a `mathematician' (or an `agent') and \emph{Opponent} who represents a `rebutter' (or an `environment'), alternately and separatedly perform moves allowed by the \emph{rules} of the game. 

On the other hand, a strategy on a game is what tells Player which move she should perform at each of her turns, i.e., `how she should play', on the game.  

\if0
As various \emph{full completeness/abstraction} results \cite{curien2007definability} in the literature have shown, game semantics has the appropriate degree of abstraction.
Another advantage of game semantics is its conceptual naturality: It models logic (resp. computation) by dialogical arguments (resp.  computational processes) between the participants of games, providing a syntax-independent explanation of syntax in a highly natural and intuitive (yet mathematically precise) manner.
Yet another advantage of game semantics is its flexibility: It models a wide range of formal languages by simply varying constraints on strategies \cite{abramsky1999game}, which enables us to compare and relate various concepts in a \emph{unified} manner.
\fi

\subsection{Concurrency and Polarization in Logic and Games}
\label{ConcurrencyAndPolarityInGameSemantics}
Problems in the game semantics of LL by Andreas Blass \cite{blass1992game} were the starting point of game semantics in its modern, \emph{categorical} form \cite{abramsky1994games}. %and a number of variants have been proposed since then.
Today, Guy McCusker's variant \cite{mccusker1998games} models ILL and IL in a \emph{unified} manner, embodying \emph{Girard's translation} $A \Rightarrow B \stackrel{\mathrm{df. }}{=} \oc A \multimap B$ of IL into ILL \cite{girard1987linear}.
Even game semantics of computation with classical features has been proposed \cite{herbelin1997games,blot2017realizability} though game semantics of CL in general has not been well-established yet.

Notably, modern game semantics of LL \cite{abramsky1999concurrent,mellies2005asynchronous} employs \emph{concurrent} games, in which more than one participant may be active simultaneously, as opposed to standard \emph{sequential} games, in which only one participant may perform a move at a time. 
Importantly, however, \emph{concurrency is not necessary at all for game semantics of CL} mentioned above. %thus, the game-semantic counterpart of classicality would not be concurrency. 

Another approach is to model the \emph{polarized} fragment of LL by \emph{polarized} (yet sequential) games \cite{laurent2002polarized}.
A game has the \emph{positive} (resp. \emph{negative}) polarity if Player (resp. Opponent) always initiates a play of the game \cite{laurent2002polarized}; standard \emph{unpolarized} games are all negative.
Also, polarization in games corresponds to polarization in logic \cite{girard1995linear,girard2011blind}, giving a unity of logic.

%Polarity in logic and games certainly gives a deeper analysis, and it induces some \emph{polarized} approaches to the unity of logic \cite{girard1993unity,laurent2002polarized}.
However, polarization is rather exotic to (game semantics of) CL; hence, it seems to have nothing to do with \emph{classicalization} of logic or games.
Also, polarization never occurs in (game semantics of) ILL, and therefore, it appears irrelevant to \emph{linear refinement} of logic. 
%For instance, Girard himself declares that although polarization is a pragmatic approach, he hesitates to give a status on it (see Sect. 12.1.1 of \cite{girard2011blind}).

\subsection{Sequential, Unpolarized Unity of Logic}
Hence, we are concerned with LL \emph{without concurrency or polarization}, which let us call \emph{linear logic negative (LL$^-$)}, and moreover, conjecture that there are mathematically precise concepts of \emph{linear refinement} and \emph{intuitionistic restriction} of logic such that the linear refinement of CL coincides with LL$^-$, and its intuitionistic restriction with the linear refinement of IL into ILL, giving a sequential, unpolarized \emph{unity} of logic (n.b., classicalization is the inverse of intuitionistic restriction).
%We set such a \emph{unity} of logic as the aim of the present work. 

Motivated in this way, we carve out the language of LL$^-$ from that of LL by discarding linear negation (for it brings polarization) and a sequent calculus \emph{LLK$^-$} for LL$^-$ from LLK by restricting the rules Cut, $\otimes$R, $\invamp$L, $\&$R, $\oplus$L and $\multimap$R (for they cause concurrency) in a certain way and adding \emph{distribution rules} to recover these rules (except $\multimap$R) and translate sequents $\Delta \vdash \Gamma$ in \emph{LK$^-$} into the sequents $\oc \Delta \vdash \wn \Gamma$ in LLK$^-$, where LK$^-$ is the calculus obtained from LK by restricting $\Rightarrow$R in the same way as $\multimap$R in LLK$^-$.
%Also, we obtain the calculus \emph{LK$^-$} from LK in the same way.
%The linear refinement of ($\Omega$)LK coincides with ($\Omega$)LLK, and the intuitionistic fragment of ($\Omega$)LLK with LLJ.
We then give a cut-elimination procedure on LLK$^-$ by \emph{normalization-by-evaluation (NBE)} \cite{berger1991inverse}, exploiting the game semantics below.

In terms of these calculi, linear refinement corresponds to eliminating exponentials $\oc$ and/or $\wn$ imposed (implicitly) in LK$^-$ and LJ, and intuitionistic restriction to limiting the number of formulas on the RHS of sequents to at most one. %where we regard $\wn A$ as the countably-infinite sequence of $A$.
In other words, LK$^-$ (resp. LJ) is obtained from LLK$^-$ (resp. LLJ) by the translation $\Delta \vdash_{\mathsf{LK^-}} \Gamma \stackrel{\mathrm{df. }}{=} \oc \Delta \vdash_{\mathsf{LLK^-}} \wn \Gamma$ (resp. $\Delta \vdash_{\mathsf{LJ}} B \stackrel{\mathrm{df. }}{=} \oc \Delta \vdash_{\mathsf{LLJ}} B$), and LLJ (resp. LJ) from LLK$^-$ (resp. LK$^-$) by intuitionistic restriction, and the two operations commute, where the subscripts indicate the underlying calculi.

\subsection{Sequential, Unpolarized Unity of Games}
We then aim to establish the game-semantic counterpart of the unity on LLK$^-$, LK$^-$, LLJ and LJ.
Let us first explain our approach in terms of \emph{categorical logic} \cite{lambek1988introduction,jacobs1999categorical}. 
Recall that the standard categorical semantics of ILL (without $\bot$ or $\oplus$) is a \emph{new-Seely category (NSC)} \cite{bierman1995categorical}, which is a \emph{symmetric monoidal closed category (SMCC)} $\mathcal{C} = (\mathcal{C}, \otimes, \top, \multimap)$ with finite products $(1, \&)$ equipped with a comonad $\oc$ and isomorphisms $\top \stackrel{\sim}{\rightarrow} \oc 1$ and $\oc A \otimes \oc B \stackrel{\sim}{\rightarrow} \oc (A \& B)$ natural in $A, B \in \mathcal{C}$ such that the canonical adjunction between $\mathcal{C}$ and the co-Kleisli category $\mathcal{C}_\oc$ is monoidal.
Its charm is its \emph{unified} semantics of ILL and IL: $\mathcal{C}_\oc$ is cartesian closed, inducing the standard semantics of IL (without $\bot$ or $\vee$) \cite{lambek1988introduction,jacobs1999categorical}.

Then, to model LL$^-$, it is a natural idea to impose on $\mathcal{C}$ another symmetric monoidal structure $(\invamp, \bot)$, finite coproducts $(0, \oplus)$, a monad $\wn$ and natural isomorphisms $\wn 0 \stackrel{\sim}{\rightarrow} \bot$ and $\wn (A \oplus B) \stackrel{\sim}{\rightarrow} \wn A \invamp \wn B$ such that the canonical adjunction between $\mathcal{C}$ and the Kleisli category $\mathcal{C}^\wn$ is monoidal.
However, it is not possible; thus, we require that $\mathcal{C}$ is equipped with a lluf subcategory $\sharp\mathcal{C}$ whose morphisms are all \emph{strict}, which has the NSC-structure inherited from $\mathcal{C}$ (except $\multimap$), finite coproducts and the triple $(\invamp, \bot, \wn)$.
Moreover, we impose a \emph{distributive law} between $\oc$ and $\wn$ \cite{power2002combining} on $\sharp\mathcal{C}$ so that the co-Kleisli and the Kleisli constructions on $\sharp\mathcal{C}$ are extended to each other, leading to the \emph{bi-Kleisli category} $\sharp\mathcal{C}_\oc^\wn \stackrel{\mathrm{df. }}{=} (\sharp\mathcal{C}_\oc)^\wn \simeq (\sharp\mathcal{C}^\wn)_\oc$.
Then, if $\sharp\mathcal{C}$ has certain natural transformations/isomorphisms, it models LLK$^-$, and $\sharp\mathcal{C}_\oc^\wn$ does LK$^-$, while $\mathcal{C}$ and $\mathcal{C}_\oc$ do LLJ and LJ.
%We shall show that the categorical semantics is not only sound but also complete by establishing term models from the calculi. 

\if0
It is perhaps worth pointing out that in our approach not all the De Morgan dualities but only $\neg A \invamp \neg B \cong \neg (A \otimes B)$ and $\wn \neg A \cong \neg \oc A$ of LL and $\neg A \vee \neg B \cong \neg (A \& B)$ of CL hold, and not \emph{small-step} but only \emph{big-step} currying is permitted, in LL$^-$ and CL.
Also, our systematic semantics suggests that coproducts in IL and (co)products in CL should be all \emph{weak}.
\fi

Finally, we instantiate the categorical semantics by a game-semantic NSC $\mathcal{LG}$ satisfying the required axioms, for which we introduce \emph{linearity} of strategies.
As the main theorem, we establish a \emph{fully complete} \cite{curien2007definability} game semantics of LLK$^-$ (without atoms) in $\sharp\mathcal{LG}$ and a game semantics of LK$^-$ in $\sharp\mathcal{LG}_\oc^\wn$.
Also, focusing on the intuitionistic part of the interpretations, we establish a fully complete game semantics of LLJ (without atoms) in the lluf subNSC $\mathcal{LG}^{\mathsf{wb}}$ of $\mathcal{LG}$, in which strategies are \emph{well-bracketed} \cite{hyland2000full}, and a game semantics of LJ in $\mathcal{LG}^{\mathsf{wb}}_\oc$. 

Thus, we establish a semantic unity of logic, where linear refinement and intuitionistic restriction correspond respectively to deletion of the co-Kleisli construction $(\_)_\oc$ and/or the Kleisli construction $(\_)^\wn$, and imposing well-bracketing on strategies.

\subsection{Our Contribution and Related Work}
Broadly, our main contribution is to establish the novel, in particular sequential and unpolarized, unity of logic in terms of sequent calculi, categories and games.
%Also, as a byproduct, we establish categorical and game semantics of CL.
Novelties are the unified (categorical and game) semantics and linearity of strategies; highlights are the full completeness results. 

Our approach stands in sharp contrast to the concurrent and/or polarized approaches \cite{abramsky1999concurrent,mellies2005asynchronous,laurent2002polarized,girard1993unity,mellies2010resource,mellies2012game} for they stick to LL or its polarized fragments, while we modify the logic into the sequential, unpolarized LL$^-$.

Our categorical account is based on the established categorical semantics of ILL \cite{seely1987linear,benton1992term,benton1993term,bierman1995categorical} and of IL \cite{lambek1988introduction,jacobs1999categorical}, as well as the study of the relation between monad and comonad  \cite{power2002combining} and its application in game semantics \cite{harmer2007categorical}.

%Finally, our constructions on games and strategies are similar to the ones proposed in the literature \cite{abramsky1999game,mccusker1998games,harmer2007categorical,laurent2002polarized}, but their details are slightly different, where note that ours are tailored for the unified treatment of the semantics. 

\subsection{Structure of the Paper}
We first present the sequent calculi in Sect.~\ref{SequentCalculi}, and the categorical semantics in Sect.~\ref{CategoricalSemantics}. 
Then, we establish the game semantics in Sect.~\ref{GameSemantics} together with some consequences in Sect.~\ref{CutElimSoundnessAndCompleteness}.
Finally, we show the full completeness in Sect.~\ref{FullCompleteness}, and draw a conclusion and propose future work in Sect.~\ref{ConclusionAndFutureWork}.

\section{Sequent Calculi for the Logics}
\label{SequentCalculi}
We assume that the reader is familiar with the formal languages and the sequent calculi for \emph{classical logic (CL)} and \emph{intuitionistic logic (IL)} \cite{gentzen1935untersuchungen,troelstra2000basic}, and those for \emph{linear logic (LL)} and \emph{intuitionistic linear logic (ILL)} \cite{girard1987linear,troelstra2000basic,abramsky1993computational}.

Throughout the paper, we focus on \emph{propositional} logic \cite{shoenfield1967mathematical}.
%Let $\mathcal{V}$ be the set of all \emph{propositional variables}.

\if0
Recall that formulas $A$ of CL /IL (resp. $B$ of LL, $C$ of ILL) are constructed by the following grammar: 
\begin{itemize}

\item $A \stackrel{\mathrm{df. }}{=} \top \mid \bot \mid A \wedge A' \mid A \vee A' \mid A \Rightarrow A'$;

\item $B \stackrel{\mathrm{df. }}{=} \top \mid \bot \mid 1 \mid 0 \mid B \otimes B' \mid B \invamp B' \mid B \& B' \mid B \oplus B' \mid B \multimap B' \mid \oc B \mid \wn B$;

\item $C \stackrel{\mathrm{df. }}{=} \top \mid \bot \mid 1 \mid C \otimes C' \mid C \& C' \mid C \oplus C' \mid C \multimap C' \mid \oc C$.

\end{itemize}
\fi

\subsection{Sequent Calculi for Classical and Intuitionistic Logics}
Let us first present our sequent calculi \emph{LK$^-$} for CL, and \emph{LJ} for IL.
Roughly, LK$^-$ is obtained from Gentzen's \emph{LK} \cite{gentzen1935untersuchungen} by restricting the rule $(\textsc{$\Rightarrow$R}) \frac{ \ \Delta, A \vdash B, \Gamma \ }{ \ \Delta \vdash A \Rightarrow B, \Gamma \ }$ to $\Rightarrow$R$^-$ given in Fig.~\ref{FigLK} (so that they can be modeled by \emph{sequential} game semantics).
%$(\textsc{Dst}) \frac{}{ \ A \wedge (B \vee C) \vdash (A \wedge B) \vee C \ }$.

As minor points, we \emph{define} negation $\neg$ by $\neg A \stackrel{\mathrm{df. }}{=} A \Rightarrow \bot$, and include \emph{top} $\top$ and the right-rule on \emph{bottom} $\bot$ for our \emph{unified} approach.
Also, we modify $\wedge$L and $\vee$R into the ones closer to the calculi for (I)LL \cite{girard1987linear,troelstra2000basic} for convenience, which, in the presence of the structural rules, does not matter. 

\begin{definition}[LK$^-$]
The calculus \emph{\bfseries LK$\boldsymbol{^-}$} for CL consists of the rules in Fig.~\ref{FigLK}.
\begin{figure}[h]
\begin{center}
\begin{align*}
&\AxiomC{}
\LeftLabel{\textsc{(Id)}}
\UnaryInfC{$A \vdash A$}
\DisplayProof \ \
\AxiomC{$\Delta \vdash B, \Gamma$}
	\AxiomC{$\Delta', B \vdash \Gamma'$}
	\LeftLabel{\textsc{(Cut)}}
\BinaryInfC{$\Delta, \Delta' \vdash \Gamma, \Gamma'$}
\DisplayProof \\
&\AxiomC{$\Delta, A, A', \Delta' \vdash \Gamma$}
\LeftLabel{\textsc{(XL)}}
\UnaryInfC{$\Delta, A', A, \Delta' \vdash \Gamma$}
\DisplayProof \ \
\AxiomC{$\Delta \vdash \Gamma, B, B', \Gamma' $}
\LeftLabel{\textsc{(XR)}}
\UnaryInfC{$\Delta \vdash \Gamma, B', B, \Gamma'$}
\DisplayProof \\
&\AxiomC{$\Delta \vdash \Gamma$} 
\LeftLabel{\textsc{(WL)}}
\UnaryInfC{$\Delta, A \vdash \Gamma$}
\DisplayProof \ \
\AxiomC{$\Delta \vdash \Gamma$}
\LeftLabel{\textsc{(WR)}}
\UnaryInfC{$\Delta \vdash B, \Gamma$}
\DisplayProof \\
&\AxiomC{$\Delta, A, A \vdash \Gamma$}
\LeftLabel{\textsc{(CL)}}
\UnaryInfC{$\Delta, A \vdash \Gamma$}
\DisplayProof \ \
\AxiomC{$\Delta \vdash B, B, \Gamma$}
\LeftLabel{\textsc{(CR)}}
\UnaryInfC{$\Delta \vdash B, \Gamma$}
\DisplayProof \\
%&\AxiomC{}
%\LeftLabel{\textsc{($\wedge \vee$Dst)}}
%\UnaryInfC{$A, (B \vee C) \vdash (A \wedge B), C$}
%\DisplayProof \ \
&\AxiomC{$\Delta \vdash \Gamma$}
\LeftLabel{\textsc{($\top$L)}}
\UnaryInfC{$\Delta, \top \vdash \Gamma$}
\DisplayProof \ \
\AxiomC{}
\LeftLabel{\textsc{($\top$R)}}
\UnaryInfC{$\vdash \top$} 
\DisplayProof \\
&\AxiomC{}
\LeftLabel{\textsc{($\bot$L)}}
\UnaryInfC{$\bot \vdash$}
\DisplayProof \ \ 
\AxiomC{$\Delta \vdash \Gamma$}
\LeftLabel{\textsc{($\bot$R)}}
\UnaryInfC{$\Delta \vdash \bot, \Gamma$}
\DisplayProof \\
&\AxiomC{$\Delta, A_1, A_2 \vdash \Gamma$}
\LeftLabel{\textsc{($\wedge$L)}}
\UnaryInfC{$\Delta, A_1 \wedge A_2 \vdash \Gamma$}
\DisplayProof \ \ 
\AxiomC{$\Delta \vdash B_1, \Gamma$}
		\AxiomC{$\Delta \vdash B_2, \Gamma$}
	\LeftLabel{\textsc{($\wedge$R)}}
\BinaryInfC{$\Delta \vdash B_1 \wedge B_2, \Gamma$} 
\DisplayProof \\
&\AxiomC{$\Delta, A_1 \vdash \Gamma$}
		\AxiomC{$\Delta, A_2 \vdash \Gamma$}
	\LeftLabel{\textsc{($\vee$L)}}
\BinaryInfC{$\Delta, A_1 \vee A_2 \vdash \Gamma$}
\DisplayProof \ \ 
\AxiomC{$\Delta \vdash B_1, B_2, \Gamma$}
\LeftLabel{\textsc{($\vee$R)}}
\UnaryInfC{$\Delta \vdash B_1 \vee B_2, \Gamma$}
\DisplayProof \\ 
&\AxiomC{$\Delta \vdash A, \Gamma$}
	\AxiomC{$\Delta, B \vdash \Gamma$}
	\LeftLabel{\textsc{($\Rightarrow$L)}}
\BinaryInfC{$\Delta, A \Rightarrow B \vdash \Gamma$}
\DisplayProof \ \ 
\AxiomC{$A \vdash B, \Gamma$}
\LeftLabel{\textsc{($\Rightarrow$R$^-$)}}
\UnaryInfC{$\vdash A \Rightarrow B, \Gamma$}
\DisplayProof
\end{align*}
\end{center}
\caption{The sequent calculus LK$^-$ for CL}
\label{FigLK}
\end{figure}
\end{definition}

In Sect.~\ref{CutElimSoundnessAndCompleteness}, we give a \emph{cut-elimination} procedure \cite{gentzen1935untersuchungen,troelstra2000basic} on LK$^-$ by \emph{normalization-by-evaluation (NBE)} \cite{berger1991inverse}, exploiting the \emph{fully complete} game semantics given in Sect.~\ref{GameSemantics}.

%It is an easy exercise to derive Cut, $\wedge$R and $\vee$L in LK$^-$, where the axiom $\wedge \vee$Dst plays a crucial role.
If one identifies sequents \emph{up to currying}, which is implicitly assumed by the \emph{one-sided} calculus for CL \cite{troelstra2000basic} and justified by the game semantics, the change of $\Rightarrow$R into $\Rightarrow$R$^-$ is not a real restriction. 
In this sense, LK$^-$ is equivalent to LK.

\begin{definition}[LJ \cite{gentzen1935untersuchungen,troelstra2000basic}]
The calculus \emph{\bfseries LJ} for IL consists of the rules of LK that have only \emph{\bfseries intuitionistic} sequents, i.e., ones such that the number of formulas on the RHS is $\leqslant 1$.
\end{definition}

%It follows from Gentzen's fundamental \emph{Hauptsatz} (on LK) \cite{gentzen1935untersuchungen} that LK$^-$ and LK derive exactly the same \emph{cut-free} proofs, and they are identical in terms of provability.
%However, there is an important difference in the level of proofs: LK$^-$ is \emph{constructive} in the sense that there is a cut-elimination procedure (given by NBE) on LK$^-$ such that some proofs are not identified modulo the cut-elimination, while the NBE is not applicable for LK (see Sect.~\ref{GameSemantics} for the details). 

%On the other hand, we have \emph{logical equivalence} between any formula $A$ and its double negation $\neg \neg A$, but they are not isomorphic (also, we do not have $\eta$-rules) for CLK$^+$, which prevents our semantics from begin trivial \cite{lambek1988introduction}.

\subsection{Sequent Calculi for Linear Logic}
In the present work, let us call the sequent calculi for LL and ILL \cite{girard1987linear,troelstra2000basic} \emph{LLK} and \emph{LLJ}, respectively.
%LLK is usually presented in the \emph{one-sided} style by exploiting the \emph{De Morgan dualities} \cite{girard1987linear}, but we prefer the two-sided style \cite{troelstra2000basic} for discarding the dualities.  
As mentioned in the introduction, LL and LLK are \emph{concurrent} and \emph{polarized}; thus, we introduce the following \emph{sequential}, \emph{unpolarized} fragment: 

\begin{notation}
Given $n \in \mathbb{N}$, we define $\overline{n} \stackrel{\mathrm{df. }}{=} \{ 1, 2, \dots, n \}$.
\end{notation}

\begin{definition}[LL$\boldsymbol{^-}$ and LLK$\boldsymbol{^-}$]
\label{DefLLKMinus}
The formal language of \emph{\bfseries linear logic negative (LL$\boldsymbol{^-}$)} is obtained from that of LL by discarding linear negation $(\_)^\bot$. %i.e., formulas $A, B$ of LL$^-$ are constructed by: 
\if0
\begin{align*}
A, B &\stackrel{\mathrm{df. }}{=} \top \mid \bot \mid 1 \mid 0 \mid A \otimes B \mid A \invamp B \mid A \& B \mid A \oplus B \mid \\
& \ \ \ \ A \multimap B \mid \oc A \mid \wn A.
\end{align*}
\fi
The calculus \emph{\bfseries LLK${\boldsymbol{^-}}$} for LL$^-$ consists of the rules in Fig.~\ref{FigLLKNegative}, where $f (A_1, A_2, \dots, A_k) \stackrel{\mathrm{df. }}{=} f A_1, f A_2, \dots, f A_k$ for all $f \in \{ \oc, \wn \}$.
\begin{figure}[h]
\begin{center}
\begin{align*}
&\AxiomC{$\Delta, \wn \oc A \vdash \Gamma$}
\LeftLabel{\textsc{($\oc\wn$L)}}
\UnaryInfC{$\Delta, \oc \wn A \vdash \Gamma$}
\DisplayProof \ \
\AxiomC{$\Delta \vdash \oc \wn B, \Gamma$}
\LeftLabel{\textsc{($\oc\wn$R)}}
\UnaryInfC{$\Delta \vdash \wn \oc B, \Gamma$}
\DisplayProof \\
&\AxiomC{$\Delta, (A \otimes B) \invamp C \vdash \Gamma$}
\LeftLabel{\textsc{($\otimes\invamp$L)}}
\UnaryInfC{$\Delta, A, B \invamp C \vdash \Gamma$}
\DisplayProof \ \
\AxiomC{$\Delta \vdash A \otimes (B \invamp C), \Gamma$}
\LeftLabel{\textsc{($\otimes\invamp$R)}}
\UnaryInfC{$\Delta \vdash A \otimes B, C, \Gamma$}
\DisplayProof \\
%&\AxiomC{$\Delta, (C \otimes  A) \oplus (C \otimes B) \vdash \Gamma$}
%\LeftLabel{\textsc{($\otimes\oplus$L)}}
%\UnaryInfC{$\Delta, C, A \oplus B \vdash \Gamma$}
%\DisplayProof \\
%&\AxiomC{$\Delta \vdash (A \invamp C) \& (B \invamp C), \Gamma$}
%\LeftLabel{\textsc{($\&\invamp$R)}}
%\UnaryInfC{$\Delta \vdash A \& B, C, \Gamma$}
%\DisplayProof \\
&\AxiomC{}
\LeftLabel{\textsc{(Id)}}
\UnaryInfC{$A \vdash A$}
\DisplayProof \ \ 
\AxiomC{$\Delta \vdash B$}
	\AxiomC{$\Delta', B \vdash \Gamma'$}
	\LeftLabel{\textsc{(Cut$^-$)}}
\BinaryInfC{$\Delta, \Delta' \vdash  \Gamma' $}
\DisplayProof \\ 
&\AxiomC{$\Delta, A, A', \Delta' \vdash \Gamma$}
\LeftLabel{\textsc{(XL)}}
\UnaryInfC{$\Delta, A', A, \Delta' \vdash \Gamma$}
\DisplayProof \ \ 
\AxiomC{$\Delta \vdash \Gamma, B, B', \Gamma'$}
\LeftLabel{\textsc{(XR)}}
\UnaryInfC{$ \Delta \vdash \Gamma, B', B, \Gamma'$}
\DisplayProof \\
&\AxiomC{$\Delta \vdash \Gamma$}
\LeftLabel{\textsc{($\oc$W)}}
\UnaryInfC{$\Delta, \oc A \vdash \Gamma$}
\DisplayProof \ \
\AxiomC{$\Delta \vdash \Gamma$}
\LeftLabel{\textsc{($\wn$W)}}
\UnaryInfC{$\Delta \vdash \wn B, \Gamma$}
\DisplayProof \\
&\AxiomC{$\Delta, \oc A, \oc A \vdash \Gamma$}
\LeftLabel{\textsc{($\oc$C)}}
\UnaryInfC{$\Delta, \oc A \vdash \Gamma$}
\DisplayProof \ \
\AxiomC{$\Delta \vdash \wn B, \wn B, \Gamma$}
\LeftLabel{\textsc{($\wn$C)}}
\UnaryInfC{$\Delta \vdash \wn B, \Gamma$}
\DisplayProof \\ 
&\AxiomC{$\Delta, A \vdash \Gamma$}
\LeftLabel{\textsc{($\oc$D)}}
\UnaryInfC{$\Delta, \oc A \vdash \Gamma$}
\DisplayProof \ \
\AxiomC{$\Delta \vdash B, \Gamma$}
\LeftLabel{\textsc{($\wn$D)}}
\UnaryInfC{$\Delta \vdash \wn B, \Gamma$}
\DisplayProof \\
&\AxiomC{$\oc \Delta, A \vdash \wn \Gamma$}
\LeftLabel{\textsc{($\wn$L)}}
\UnaryInfC{$\oc \Delta, \wn A \vdash \wn \Gamma$}
\DisplayProof \ \ 
\AxiomC{$\oc \Delta \vdash B, \wn \Gamma$}
\LeftLabel{\textsc{($\oc$R)}}
\UnaryInfC{$\oc \Delta \vdash \oc B, \wn \Gamma$}
\DisplayProof \\
&\AxiomC{}
\LeftLabel{\textsc{($0$L)}}
\UnaryInfC{$0 \vdash \Gamma$}
\DisplayProof \ \ 
\AxiomC{}
\LeftLabel{\textsc{($1$R)}}
\UnaryInfC{$\Delta \vdash 1, \Gamma$}
\DisplayProof \\
&\AxiomC{$\Delta \vdash \Gamma$}
\LeftLabel{\textsc{($\top$L)}}
\UnaryInfC{$ \Delta, \top \vdash \Gamma$}
\DisplayProof \ \
\AxiomC{}
\LeftLabel{\textsc{($\top$R)}}
\UnaryInfC{$\vdash \top$}
\DisplayProof \ \
\AxiomC{}
\LeftLabel{\textsc{($\bot$L)}}
\UnaryInfC{$\bot \vdash$}
\DisplayProof \ \
\AxiomC{$\Delta \vdash \Gamma$}
\LeftLabel{\textsc{($\bot$R)}}
\UnaryInfC{$\Delta \vdash \bot, \Gamma$}
\DisplayProof \\
&\AxiomC{$\Delta, A_1, A_2 \vdash \Gamma$}
\LeftLabel{\textsc{($\otimes$L)}}
\UnaryInfC{$\Delta, A_1 \otimes A_2 \vdash \Gamma$}
\DisplayProof \ \ 
\AxiomC{$\Delta_1\vdash B_1$}
		\AxiomC{$\Delta_2 \vdash B_2$}
	\LeftLabel{\textsc{($\otimes$R$^-$)}}
\BinaryInfC{$\Delta_1, \Delta_2 \vdash B_1 \otimes B_2$}
\DisplayProof \\ 
&\AxiomC{$\Delta, A_i \vdash \Gamma$}
		\AxiomC{$i \in \overline{2}$}
	\LeftLabel{\textsc{($\&$L)}}
\BinaryInfC{$\Delta, A_1 \& A_2 \vdash \Gamma$}
\DisplayProof \ \
\AxiomC{$\Delta \vdash B_1$}
		\AxiomC{$\Delta \vdash B_2$}
	\LeftLabel{\textsc{($\&$R$^-$)}}
\BinaryInfC{$\Delta \vdash B_1 \& B_2$}
\DisplayProof \\ 
&\AxiomC{$A_1 \vdash \Gamma_1$}
		\AxiomC{$A_2 \vdash \Gamma_2$}
	\LeftLabel{\textsc{($\invamp$L$^-$)}}
\BinaryInfC{$A_1 \invamp A_2 \vdash \Gamma_1, \Gamma_2$}
\DisplayProof \ \
\AxiomC{$\Delta \vdash B_1, B_2, \Gamma$}
\LeftLabel{\textsc{($\invamp$R)}}
\UnaryInfC{$\Delta \vdash B_1 \invamp B_2, \Gamma$}
\DisplayProof \\
&\AxiomC{$A_1 \vdash \Gamma$}
		\AxiomC{$A_2 \vdash \Gamma$}
	\LeftLabel{\textsc{($\oplus$L$^-$)}}
\BinaryInfC{$A_1 \oplus A_2 \vdash \Gamma$}
\DisplayProof \ \
\AxiomC{$\Delta \vdash B_i, \Gamma$}
		\AxiomC{$i \in \overline{2}$}
	\LeftLabel{\textsc{($\oplus$R)}}
\BinaryInfC{$\Delta \vdash B_1 \oplus B_2, \Gamma$}
\DisplayProof \\ 
&\AxiomC{$\Delta_1 \vdash A, \Gamma_1$}
	\AxiomC{$\Delta_2, B \vdash \Gamma_2$}
	\LeftLabel{\textsc{($\multimap$L)} }
\BinaryInfC{$\Delta_1, \Delta_2, A \multimap B \vdash \Gamma_1, \Gamma_2$}
\DisplayProof \ \ 
\AxiomC{$A \vdash B, \Gamma$}
\LeftLabel{\textsc{($\multimap$R$^-$)}}
\UnaryInfC{$\vdash A \multimap B, \Gamma$}
\DisplayProof 
%&\frac{ \ \ }{ \ (A \otimes B) \invamp (A \otimes B) \invamp (C \otimes D) \invamp (C \otimes D) \dashv \vdash (A \invamp C) \otimes (A \invamp D) \otimes (B \invamp C) \otimes (B \invamp D) \ } \textsc{($\otimes$Exp)} \\
%&\frac{ \ \ }{ \ (A \oplus B) \invamp (A \oplus B) \invamp (C \oplus D) \invamp (C \oplus D) \dashv \vdash (A \invamp C) \oplus (A \invamp D) \oplus (B \invamp C) \oplus (B \invamp D) \ } \textsc{($\oplus$Exp)}
\end{align*}
\end{center}
\caption{The sequent calculus LLK$^-$ for LL$^-$}
\label{FigLLKNegative}
\end{figure}
\end{definition}

\if0
\begin{remark}
We switch between the traditional notations for \emph{top} and \emph{one} \cite{girard1987linear} for our \emph{unity} of logic: In terms of categorical logic in Sect.~\ref{CategoricalSemantics}, one and zero are terminal and initial objects, and top and bottom are units w.r.t. $\otimes$ and $\invamp$, respectively.
\end{remark}
\fi

That is, LLK$^-$ is obtained from LLK by discarding linear negation $(\_)^\bot$, restricting the rules $(\textsc{Cut}) \frac{ \ \Delta \vdash B, \Gamma \ \ \ \ \Delta', B \vdash \Gamma' \ }{ \ \Delta, \Delta' \vdash \Gamma, \Gamma' \ }$, $(\textsc{$\otimes$R}) \frac{ \ \Delta_1 \vdash B_1, \Gamma_1 \ \ \ \ \Delta_2 \vdash B_2, \Gamma_2 \ }{ \ \Delta_1, \Delta_2 \vdash B_1 \otimes B_2, \Gamma_1, \Gamma_2 \ }$, $(\textsc{$\invamp$L}) \frac{ \ \Delta_1, A_1 \vdash \Gamma_1 \ \ \ \ \Delta_2, A_2 \vdash \Gamma_2 \ }{ \ \Delta_1, \Delta_2, A_1 \invamp A_2 \vdash \Gamma_1, \Gamma_2 \ }$, $(\textsc{$\&$R}) \frac{ \ \Delta \vdash B_1, \Gamma \ \ \ \ \Delta \vdash B_2, \Gamma \ }{ \ \Delta \vdash B_1 \& B_2, \Gamma \ }$, $(\textsc{$\oplus$L}) \frac{ \ \Delta, A_1 \vdash \Gamma \ \ \ \ \Delta, A_2 \vdash \Gamma \ }{ \ \Delta, A_1 \oplus A_2 \vdash \Gamma \ }$ and $(\textsc{$\multimap$R}) \frac{ \ \Delta, A \vdash B, \Gamma \ }{ \ \Delta \vdash A \multimap B, \Gamma \ }$, respectively, to Cut$^-$, $\otimes$R$^-$, $\invamp$L$^-$, $\&$R$^-$, $\oplus$L$^-$ and $\multimap$R$^-$ given in Fig.~\ref{FigLLKNegative}, and adding the \emph{\bfseries distribution rules} $\oc\wn$L, $\oc\wn$R, $\otimes\invamp$L and $\otimes\invamp$R.
It is easy to see that Cut, $\otimes$R and $\invamp$L are derivable in LLK$^-$ in the presence of $\otimes\invamp$L and $\otimes\invamp$R.
For example, Cut is derived in LLK$^-$ by:
\begin{equation*}
\AxiomC{}
\doubleLine
\UnaryInfC{$\Delta' \vdash \otimes \Delta'$}
			\AxiomC{$\Delta \vdash B, \Gamma$} 
			\doubleLine
			\UnaryInfC{$\Delta \vdash \invamp (B, \Gamma)$}
		\BinaryInfC{$\Delta', \Delta \vdash (\otimes \Delta') \otimes (\invamp (B, \Gamma))$}
		\UnaryInfC{$\Delta', \Delta \vdash \otimes (\Delta', B) \invamp (\invamp \Gamma)$}
				\AxiomC{$\Delta', B \vdash \Gamma'$}
				\doubleLine
				\UnaryInfC{$\otimes (\Delta', B) \vdash \Gamma'$}
						\AxiomC{}
						\doubleLine
						\UnaryInfC{$\invamp \Gamma \vdash \Gamma$}
					\BinaryInfC{$\otimes (\Delta', B) \invamp (\invamp \Gamma) \vdash \Gamma', \Gamma$}
			\BinaryInfC{$\Delta', \Delta \vdash \Gamma', \Gamma$}
			\doubleLine
			\UnaryInfC{$\Delta, \Delta' \vdash \Gamma, \Gamma'$}
		\DisplayProof
\end{equation*}
where the double line indicates a multiple application of rules, $\otimes \boldsymbol{\epsilon} \stackrel{\mathrm{df. }}{=} \top$, $\otimes (\Delta, A) \stackrel{\mathrm{df. }}{=} (\otimes \Delta) \otimes A$, $\invamp \boldsymbol{\epsilon} \stackrel{\mathrm{df. }}{=} \bot$ and $\invamp (\Gamma, B) \stackrel{\mathrm{df. }}{=} (\invamp \Gamma) \invamp B$.
These \emph{derived} Cut, $\otimes$R and $\invamp$L faithfully represent our categorical semantics given in Sect.~\ref{CategoricalSemantics}.

On the other hand, the remaining two distribution rules $\oc\wn$L and $\oc\wn$R enable us to translate LK$^-$ into LLK$^-$:
\begin{theorem}[Translation of LK$^-$ into LLK$^-$]
\label{ThmTranslationOfCLIntoLL}
There is a translation $\mathscr{T}_{\mathsf{c}}$ of formulas and proofs that assigns, to every proof $p$ of a sequent $\Delta \vdash \Gamma$ in LK$^-$, a proof $\mathscr{T}_{\mathsf{c}}(p)$ of a sequent $\oc \mathscr{T}_{\mathsf{c}}^\ast(\Delta) \vdash \wn \mathscr{T}_{\mathsf{c}}^\ast(\Gamma)$ in LLK$^-$, where $\mathscr{T}_{\mathsf{c}}(\top) \stackrel{\mathrm{df. }}{=} \top$, $\mathscr{T}_{\mathsf{c}}(\bot) \stackrel{\mathrm{df. }}{=} \bot$, $\mathscr{T}_{\mathsf{c}}(A \wedge B) \stackrel{\mathrm{df. }}{=} \wn \mathscr{T}_{\mathsf{c}}(A) \& \wn \mathscr{T}_{\mathsf{c}}(B)$, $\mathscr{T}_{\mathsf{c}}(A \vee B) \stackrel{\mathrm{df. }}{=} \oc \mathscr{T}_{\mathsf{c}}(A) \oplus \oc \mathscr{T}_{\mathsf{c}}(B)$ and $\mathscr{T}_{\mathsf{c}}(A \Rightarrow B) \stackrel{\mathrm{df. }}{=} \oc \wn \mathscr{T}_{\mathsf{c}}(A) \multimap \wn \oc \mathscr{T}_{\mathsf{c}}(B)$.
\end{theorem}
\begin{proof}
We shall translate each rule of LK$^-$ into a proof tree in LLK$^-$.
%By induction on $\Delta \vdash \Gamma$.
%For brevity, we omit $\mathscr{T}_{\mathsf{c}}$ on formulas.
First, note that we have shown that Cut is admissible in LLK$^-$; then, Cut of LK$^-$ is translated into LLK$^-$ by:
\begin{equation*}
\AxiomC{$\oc \Delta \vdash \wn B, \wn \Gamma$}
%\LeftLabel{\textsc{($\oc$R)}}
\UnaryInfC{$\oc \Delta \vdash \oc \wn B, \wn \Gamma$}
\UnaryInfC{$\oc \Delta \vdash \wn \oc B, \wn \Gamma$}
		\AxiomC{$\oc \Delta', \oc B \vdash \wn \Gamma'$}
		%\LeftLabel{\textsc{($\wn$L)}}
		\UnaryInfC{$\oc \Delta', \wn \oc B \vdash \wn \Gamma'$}
		\doubleLine
	\BinaryInfC{$\oc \Delta, \oc \Delta' \vdash \wn \Gamma, \wn \Gamma'$}
	\DisplayProof
\end{equation*}

WL, WR, CL, CR, XL and XR of LK$^-$ are translated, by $\oc$W, $\wn$W, $\oc$C, $\wn$C, XL and XR of LLK$^-$, respectively, and Id of LK$^-$ by Id, $\oc$D and $\wn$D of LLK$^-$, in the obvious manner.

$\top$L and $\top$R of LK$^-$ are translated by: 
\begin{equation*}
\AxiomC{$\oc \Delta \vdash \wn \Gamma$}
\UnaryInfC{$\oc \Delta, \top \vdash \wn \Gamma$}
\UnaryInfC{$\oc \Delta, \oc \top \vdash \wn \Gamma$} \
\DisplayProof \ \ \ \ 
\AxiomC{$\vdash \top$}
\UnaryInfC{$\vdash \wn \top$}
\DisplayProof
\end{equation*}
respectively into LLK$^-$; $\bot$L and $\bot$R are symmetric.
%$\textsc{($\oc$D)} \dfrac{\textsc{($\bot$L)} \dfrac{}{ \ \bot \vdash \ }}{ \ \oc \bot \vdash \ }$, $\textsc{($\wn$D)} \dfrac{\textsc{($\bot$R)} \dfrac{ \ \oc \Delta \vdash \wn \Gamma \ }{ \ \oc \Delta \vdash \bot, \wn \Gamma \ }}{ \ \oc \Delta \vdash \wn \bot, \wn \Gamma \ }$.

$\wedge$L of LK$^-$ is translated into LLK$^-$ by: 
\begin{equation*}
\AxiomC{$\oc \Delta, \oc A_1, \oc A_2 \vdash \wn \Gamma$}
\doubleLine
\UnaryInfC{$\oc \Delta, \wn \oc A_1, \wn \oc A_2 \vdash \wn \Gamma$}
\doubleLine
\UnaryInfC{$\oc \Delta, \oc \wn A_1, \oc \wn A_2 \vdash \wn \Gamma$}
\UnaryInfC{$\oc \Delta, \oc \wn A_1 \otimes \oc \wn A_2 \vdash \wn \Gamma$}
\LeftLabel{\textsc{(Sub$^{\otimes, \&}_{\wn A_1, \wn A_2}$)}}
\UnaryInfC{$\oc \Delta, \oc (\wn A_1 \& \wn A_2) \vdash \wn \Gamma$}
\DisplayProof
\end{equation*}
where Sub$^{\otimes, \&}_{X, Y}$ is Cut$^-$ with:
\begin{equation*}
\AxiomC{$X \vdash X$}
\UnaryInfC{$X \& Y \vdash X$}
\UnaryInfC{$\oc (X \& Y) \vdash X$}
\UnaryInfC{$\oc (X \& Y) \vdash \oc X$}
\AxiomC{$Y \vdash Y$}
\UnaryInfC{$X \& Y \vdash Y$}
\UnaryInfC{$\oc (X \& Y) \vdash Y$}
\UnaryInfC{$\oc (X \& Y) \vdash \oc Y$}
\BinaryInfC{$\oc (X \& Y), \oc (X \& Y) \vdash \oc X \otimes \oc Y$}
\UnaryInfC{$\oc (X \& Y) \vdash \oc X \otimes \oc Y$}
\DisplayProof
\end{equation*}

Next, it is not hard to translate $\textsc{($\wedge\vee$L)}\frac{ \ \Delta, (A \wedge B) \vee C \vdash \Gamma \ }{ \ \Delta, A \wedge (B \vee C) \vdash \Gamma \ }$ and $\textsc{($\wedge\vee$R)}\frac{ \ \Delta \vdash A \wedge (B \vee C), \Gamma \ }{ \ \Delta \vdash (A \wedge B) \vee C, \Gamma \ }$ into LLK$^-$ (by $\otimes\invamp$L, $\otimes\invamp$R, $\wn$L and $\oc$R); we omit the details for lack of space.
Thus, translations of $\wedge$R and $\vee$L are reduced to those of $\textsc{($\wedge$R$^-$)}\frac{ \ \Delta \vdash B_1 \ \ \ \ \Delta \vdash B_2 \ }{ \ \Delta \vdash B_1 \wedge B_2 \ }$ and $\textsc{($\vee$L$^-$)}\frac{ \ A_1 \vdash \Gamma \ \ \ \ A_2 \vdash \Gamma \ }{ \ A_1 \vee A_2 \vdash \Gamma \ }$, respectively, in the obvious way. 

Then, a translation of $\wedge$R$^-$ is very simple: 
\begin{equation*}
\AxiomC{$\oc \Delta \vdash \wn B_1$}
\AxiomC{$\oc \Delta \vdash \wn B_2$}
\BinaryInfC{$\oc \Delta \vdash \wn B_1 \& \wn B_2$}
\UnaryInfC{$\oc \Delta \vdash \wn (\wn B_1 \& \wn B_2)$}
\DisplayProof
\end{equation*}

Similarly, $\vee$L$^-$ is translated by: 
\begin{equation*}
\AxiomC{$\oc A_1 \vdash \wn \Gamma$}
\AxiomC{$\oc A_2 \vdash \wn \Gamma$}
\BinaryInfC{$\oc A_1 \oplus \oc A_2 \vdash \wn \Gamma$}
\UnaryInfC{$\oc (\oc A_1 \oplus \oc A_2) \vdash \wn \Gamma$}
\DisplayProof
\end{equation*}

Also, $\vee$R is translated by:
\begin{equation*}
\AxiomC{$\oc \Delta \vdash \wn B_1, \wn B_2, \wn \Gamma$}
\doubleLine
\UnaryInfC{$\oc \Delta, \vdash \oc \wn B_1, \oc \wn B_2, \wn \Gamma$}
\doubleLine
\UnaryInfC{$\oc \Delta, \vdash \wn \oc B_1, \wn \oc B_2, \wn \Gamma$}
\doubleLine
\UnaryInfC{$\oc \Delta, \vdash (\wn \oc B_1 \invamp \wn \oc B_2) \invamp \wn \Gamma$}
\LeftLabel{\textsc{(Sub$^{\oplus, \invamp}_{\oc B_1, \oc B_2} \invamp \wn \Gamma$)}}
\UnaryInfC{$\oc \Delta \vdash \wn (\oc B_1 \oplus \oc B_2), \wn \Gamma$}
\DisplayProof
\end{equation*}
where Sub$^{\oplus, \invamp}_{X, Y} \invamp Z$ is Cut$^-$ with:
\begin{equation*}
\AxiomC{$X \vdash X$}
\UnaryInfC{$X \vdash X \oplus Y$}
\UnaryInfC{$X \vdash \wn (X \oplus Y)$}
\UnaryInfC{$\wn X \vdash \wn (X \oplus Y)$}
		\AxiomC{$Y \vdash Y$}
		\UnaryInfC{$Y \vdash X \oplus Y$}
		\UnaryInfC{$Y \vdash \wn (X \oplus Y)$}
		\UnaryInfC{$\wn Y \vdash \wn (X \oplus Y)$}
	\BinaryInfC{$\wn X \invamp \wn Y \vdash \wn (X \oplus Y), \wn (X \oplus Y)$}
	\UnaryInfC{$\wn X \invamp \wn Y \vdash \wn (X \oplus Y)$}
		\AxiomC{$Z \vdash Z$}
	\BinaryInfC{$(\wn X \invamp \wn Y) \invamp Z \vdash \wn (X \oplus Y), Z$}
\DisplayProof
\end{equation*}

Next, $\Rightarrow$L is translated by:
\begin{equation*}
\AxiomC{$\oc \Delta \vdash \wn A, \wn \Gamma$}
\UnaryInfC{$\oc \Delta \vdash \oc \wn A, \wn \Gamma$}
\AxiomC{$\oc \Delta, \oc B \vdash \wn \Gamma$}
\UnaryInfC{$\oc \Delta, \wn \oc B \vdash \wn \Gamma$}
\BinaryInfC{$\oc \Delta, \oc \Delta, \oc \wn A \multimap \wn \oc B \vdash \wn \Gamma, \wn \Gamma$}
\doubleLine
\UnaryInfC{$\oc \Delta, \oc (\oc \wn A \multimap \wn \oc B) \vdash \wn \Gamma$}
\DisplayProof
\end{equation*}
and $\Rightarrow$R$^-$ by:
\begin{equation*}
\AxiomC{$\oc A \vdash \wn B, \wn \Gamma$}
\UnaryInfC{$\wn \oc A \vdash \wn B, \wn \Gamma$}
\UnaryInfC{$\oc \wn A \vdash \wn B, \wn \Gamma$}
\UnaryInfC{$\oc \wn A \vdash \oc \wn B, \wn \Gamma$}
\UnaryInfC{$\oc \wn A \vdash \wn \oc B, \wn \Gamma$}
\UnaryInfC{$\vdash \oc \wn A \multimap \wn \oc B, \wn \Gamma$}
\UnaryInfC{$\vdash \wn (\oc \wn A \multimap \wn \oc B), \wn \Gamma$}
\DisplayProof
\end{equation*}
which completes the proof.
\if0
Finally, the rule $\wedge \vee$Dst is translated into LLK$^-$ by:
\begin{equation*}
\AxiomC{}
\UnaryInfC{$A \& (B \oplus C) \vdash (A \& B) \oplus C$}
\doubleLine
\UnaryInfC{$A \& (\oc B \oplus \oc C) \vdash (\wn A \& \wn B) \oplus C$}
\doubleLine
\UnaryInfC{$\oc (A \& (\oc B \oplus \oc C)) \vdash \wn ((\wn A \& \wn B) \oplus C)$}
\doubleLine
\UnaryInfC{$\oc A, \oc (\oc B \oplus \oc C) \vdash \wn (\wn A \& \wn B), \wn C$}
\DisplayProof
\end{equation*} 
where each of the double-lined inferences is by appropriate applications of Cut$^-$, which should be obvious at this point, and the details are left to the reader, completing the proof.
\fi
\end{proof}

The translation $\mathscr{T}_{\mathsf{c}}$ of Thm.~\ref{ThmTranslationOfCLIntoLL} is, as far as we are concerned, a novel one.
In contrast to the translations of CL into LL given in \cite{girard1987linear,girard1995linear,girard1993unity,laurent2002polarized}, our translation is \emph{unpolarized}.
%We shall abstract, categorically, this syntactic translation in Sect.~\ref{CategoricalSemantics}. %and prove that the translation is \emph{full} and \emph{faithful} in Sect~\ref{CutElimSoundnessAndCompleteness}.

\if0
Note, in particular, that the translation $\mathscr{T}_{\mathsf{c}}$ of implication $\Rightarrow$ suggests that each sequent $\Delta \vdash \Gamma$ in LK$^-$ is to be translated into the sequent $\oc \wn \Delta \vdash \wn \oc \Gamma$ in LLK$^-$ (n.b., if we had defined $\mathscr{T}_{\mathsf{c}}(A \Rightarrow B) \stackrel{\mathrm{df. }}{=} \oc A \multimap \wn B$, then we could not translate the rule $\Rightarrow$L).
However, proofs of a sequent $\oc \wn \Delta \vdash \wn \oc \Gamma$ are in a one-to-one correspondence with those of the sequent $\oc \Delta \vdash \wn \Gamma$ in LLK$^-$, which we prove in Sect.~\ref{CutElimSoundnessAndCompleteness}, and their game-semantic interpretations essentially coincide in Sect.~\ref{GameSemantics}. 
Thus, we have adopted the simpler translation $\Delta \vdash_{\mathsf{LK}^-} \Gamma \mapsto \oc \Delta \vdash_{\mathsf{LLK}^-} \wn \Gamma$.
\fi

Finally, note that the following standard result (Thm.~\ref{ThmTranslationOfILIntoILL}) can be seen as the \emph{intuitionistic restriction} of Thm.~\ref{ThmTranslationOfCLIntoLL} (except $\vee$):
\begin{definition}[LLJ \cite{abramsky1993computational,mellies2009categorical}]
\label{DefLLKMinus}
The formal language of ILL is obtained from that of LL by discarding $(\_)^\bot$, $\wn$ and $\invamp$.
The calculus \emph{\bfseries LLJ} for ILL consists of the rules of LLK that have only intuitionistic sequents.
\end{definition}

\begin{theorem}[Translation of LJ into LLJ \cite{girard1987linear,girard1995linear}]
\label{ThmTranslationOfILIntoILL}
There is a translation $\mathscr{T}_{\mathsf{i}}$ of formulas and proofs that assigns, to every proof $p$ of a sequent $\Delta \vdash B$ in LJ, a proof $\mathscr{T}_{\mathsf{i}}(p)$ of a sequent $\oc \mathscr{T}_{\mathsf{i}}^\ast(\Delta) \vdash \mathscr{T}_{\mathsf{i}}^\ast(\Gamma)$ in LLJ, where $\mathscr{T}_{\mathsf{i}}(\top) \stackrel{\mathrm{df. }}{=} \top$, $\mathscr{T}_{\mathsf{i}}(\bot) \stackrel{\mathrm{df. }}{=} \bot$, $\mathscr{T}_{\mathsf{i}}(A \wedge B) \stackrel{\mathrm{df. }}{=} \mathscr{T}_{\mathsf{i}}(A) \& \mathscr{T}_{\mathsf{i}}(B)$, $\mathscr{T}_{\mathsf{i}}(A \vee B) \stackrel{\mathrm{df. }}{=} \oc \mathscr{T}_{\mathsf{i}}(A) \oplus \oc \mathscr{T}_{\mathsf{i}}(B)$ and $\mathscr{T}_{\mathsf{i}}(A \Rightarrow B) \stackrel{\mathrm{df. }}{=} \oc \mathscr{T}_{\mathsf{i}}(A) \multimap \mathscr{T}_{\mathsf{i}}(B)$.
\end{theorem}

\begin{remark}
Note that $\mathscr{T}_{\mathsf{c}}$ translates $\vee$L in terms of $\oplus$L$^-$ by utilizing distribution rules, while $\mathscr{T}_{\mathsf{i}}$ translates $\vee$L in terms of $\oplus$L.
Nevertheless, except the mismatch between the translations of $\vee$L, $\mathscr{T}_{\mathsf{i}}$ can be seen as the intuitionistic restriction of $\mathscr{T}_{\mathsf{c}}$.
\end{remark}

%Henceforth, we exclude $\vee$ from LJ, and $\oplus$ from LLJ.

\section{Categorical Semantics}
\label{CategoricalSemantics}
Next, we proceed to give \emph{categorical semantics} \cite{lambek1988introduction,jacobs1999categorical} of the sequent calculi introduced in Sect.~\ref{SequentCalculi} in a \emph{unified} manner. 

We assume that the reader is familiar with the basic concepts of \emph{\bfseries symmetric monoidal (closed) categories (SM(C)Cs)} and \emph{\bfseries monoidal adjoints} \cite{mac2013categories,bierman1995categorical}.
%We shall establish our categorical logic based on these concepts.
To indicate what is to be modeled, we frequently employ notations from LL for categorical structures in this section. 
Also, we often do not specify natural isomorphisms even if they are part of a categorical structure. 

\begin{remark}
Since cut-eliminations on the calculi are given in Sect.~\ref{CutElimSoundnessAndCompleteness}, we postpone (equational) soundness/completeness of the semantics to Sect.~\ref{CutElimSoundnessAndCompleteness}, and in this section just assign objects and morphisms to formulas and proofs, respectively.
\end{remark}

\subsection{Categorical Semantics of ILL and IL}
Let us first recall the standard categorical semantics of LLJ (without $\bot$ or $\oplus$), introducing a nonstandard terminology: 
\begin{definition}[BwLSMCs]
\label{DefBwLSMCs}
A SMC $\mathcal{C} = (\mathcal{C}, \otimes, \top)$ is \emph{\bfseries backward liberalizable (BwL)} if it has finite products $(1, \&)$ and is equipped with:
\begin{itemize}

\item A comonad $\oc = (\oc, \epsilon, \delta)$ on $\mathcal{C}$ such that the canonical adjunction between $\mathcal{C}$ and the co-Kleisli category $\mathcal{C}_\oc$ of $\mathcal{C}$ over $\oc$ is monoidal;

\item Isomorphisms $\top \stackrel{\sim}{\rightarrow} \oc 1$ and $\oc A \otimes \oc B \stackrel{\sim}{\rightarrow} \oc (A \& B)$ natural in $A, B \in \mathcal{C}$.

\end{itemize}
\end{definition}

In other words, a BwLSMC is simply a \emph{\bfseries new-Seely category (NSC)} \cite{bierman1995categorical} without a \emph{closed} structure $\multimap$; it is just to state Thm.~\ref{LemLLNegativeWithoutDst} and Def.~\ref{DefBiLSMCCs} concisely.
Recall that NSCs give a (equationally sound and complete) semantics of ILL without $\bot$ or $\oplus$ (w.r.t. the term calculus given in \cite{benton1992term,benton1993term}): 
\begin{theorem}[Semantics of ILL without $\boldsymbol{\bot}$ or $\boldsymbol{\oplus}$ \cite{bierman1995categorical}]
NSCs give a semantics of LLJ without $\bot$ or $\oplus$.
\end{theorem}
%For completeness, however, we need to show that the term calculus gives rise to a BwLSMCC, which we leave as future work.

Recall that a strong advantage of NSCs is the following: 
\begin{theorem}[CCCs via NSCs \cite{seely1987linear,bierman1995categorical}]
\label{CoroCoprodWeakProdInFwLSMCs}
The co-Kleisli category $\mathcal{C}_\oc$ of a NSC $\mathcal{C}$ over the equipped comonad $\oc$ is cartesian closed.
\end{theorem}
\begin{proof}[Proof (sketch)]
Let $\mathcal{C} = (\mathcal{C}, \otimes, \top, \multimap, \oc)$ be a NSC, and $A, B \in \mathcal{C}_\oc$.
%We actually do not need all the components of a given BLSMCC $\mathcal{C}$; we only need that $\mathcal{C}$ is a \emph{Seely category (SC)} \cite{seely1987linear}.
First, we may give, as a terminal object and a binary product of $A$ and $B$ in $\mathcal{C}_\oc$, a terminal object $1$ and a diagram $A \stackrel{\epsilon_{A \& B} ; \pi_1}{\leftarrow} \oc (A \& B) \stackrel{\epsilon_{A \& B} ; \pi_2}{\rightarrow} B$ in $\mathcal{C}$, respectively.

Next, we may give $\oc A \multimap B \in \mathcal{C}$ as an exponential object $A \Rightarrow B$ from $A$ to $B$ in $\mathcal{C}_\oc$.
In fact, we have an isomorphism:
\begin{align*}
\mathcal{C}_\oc(A \& B, C) &= \mathcal{C}(\oc (A \& B), C) \\
&\cong \mathcal{C}(\oc A \otimes \oc B, C) \\
&\cong \mathcal{C}(\oc A, \oc B \multimap C) \\
&= \mathcal{C}_\oc (A, B \Rightarrow C)
\end{align*}
natural in $A, C \in \mathcal{C}$.
\end{proof}

The \emph{linear decomposition} $A \Rightarrow B \stackrel{\mathrm{df. }}{=} \oc A \multimap B$ of exponential objects in $\mathcal{C}_\oc$ into the comonad $\oc$ and the closed structure $\multimap$ in $\mathcal{C}$ is the categorical counterpart of \emph{Girard's translation} \cite{girard1987linear}, and it gives a \emph{unified} semantics of ILL and IL, where note that CCCs give the standard categorical semantics of IL (without $\bot$ or $\vee$) \cite{lambek1988introduction,jacobs1999categorical}.
What about $\bot$ and $\vee$?

It then seems a natural idea to add finite coproducts $(0, \oplus)$ to the NSC $\mathcal{C}$.
As pointed out in \cite{seely1987linear}, however, finite coproducts in $\mathcal{C}$ become \emph{weak} in $\mathcal{C}_\oc$: The morphism $\oc 0 \stackrel{\epsilon_0}{\rightarrow} 0 \dashrightarrow A$ in $\mathcal{C}$ is a morphism $0 \rightarrow A$ in $\mathcal{C}_\oc$ for each $A \in \mathcal{C}$, but it may not be unique for $\oc 0$ is not necessarily initial in $\mathcal{C}$; also, it seems reasonable to take, as a coproduct of $A, B \in \mathcal{C}_\oc$ in $\mathcal{C}_\oc$, a coproduct $\oc A \stackrel{\iota_1}{\rightarrow} \oc A \oplus \oc B \stackrel{\iota_2}{\leftarrow} \oc B$ in $\mathcal{C}$, but the induced copairings in $\mathcal{C}_\oc$ do not necessarily satisfy uniqueness as they may not be copairings in $\mathcal{C}$. 
Meanwhile, this construction clearly works for \emph{weak} coproducts in $\mathcal{C}$ as well, which is important as the game-semantic NSC in Sect.~\ref{GameSemantics} has only weak ones.
To summarize:
\begin{corollary}[Semantics of ILL and IL \cite{seely1987linear,bierman1995categorical}]
\label{CoroCategoricalSemanticsOfIL}
A NSC $\mathcal{C} = (\mathcal{C}, \otimes, \top, \multimap, \oc)$ with weak finite coproducts $(0, \oplus)$ gives a semantics of LLJ.
Moreover, the co-Kleisli category $\mathcal{C}_\oc$ has:
\begin{itemize}

\item Finite products just given by finite products $(1, \&)$ in $\mathcal{C}$;

\item Exponential objects given by $A \Rightarrow B \stackrel{\mathrm{df. }}{=} \oc A \multimap B$ in $\mathcal{C}$ for all $A, B \in \mathcal{C}_\oc$;

\item Weak finite coproducts given by $(0, \oc (\_) \oplus \oc (\_))$ in $\mathcal{C}$

\end{itemize}
and thus, $\mathcal{C}_\oc$ gives a semantics of LJ \cite{lambek1988introduction,jacobs1999categorical}.
\end{corollary}

Note that the derivation of the categorical semantics of IL in $\mathcal{C}_\oc$ from that of ILL in $\mathcal{C}$ coincides with the translation $\mathscr{T}_{\mathsf{i}}$.

\subsection{Categorical Semantics of LL$\mathbf{^-}$ and CL}
Our main idea on modeling LLK$^-$ is then to introduce the following symmetric structure to BwLSMCs: 
\begin{definition}[FwLSMCs]
A SMC $\mathcal{C} = (\mathcal{C}, \invamp, \bot)$ is \emph{\bfseries forward-liberalizable (FwL)} if it has finite coproducts $(0, \oplus)$ and is equipped with:
\begin{itemize}

\item A monad $\wn = (\wn, \eta, \mu)$ on $\mathcal{C}$ such that the canonical adjunction between $\mathcal{C}$ and the Kleisli category $\mathcal{C}^\wn$ of $\mathcal{C}$ over $\wn$ is monoidal;

\item Isomorphisms $\wn 0 \stackrel{\sim}{\rightarrow} \bot$ and $\wn (A \oplus B) \stackrel{\sim}{\rightarrow} \wn A \invamp \wn B$ natural in $A, B \in \mathcal{C}$.

\end{itemize}
\end{definition} 

\begin{corollary}[Coproducts and weak products in FwLSMCs]
\label{CoroCoprodWeakProdInFwLSMCs}
The Kleisli category $\mathcal{C}^\wn$ of a FwLSMC $\mathcal{C} = (\mathcal{C}, \invamp, \bot, \wn)$ with weak finite products $(1, \&)$ has:
\begin{itemize}

\item Finite coproducts given by finite coproducts $(0, \oplus)$ in $\mathcal{C}$;

\item Weak finite products given by $(1, \wn (\_) \& \wn(\_))$ in $\mathcal{C}$.

\end{itemize}
\end{corollary}
\begin{proof}
Symmetric to Cor.~\ref{CoroCategoricalSemanticsOfIL}.
\end{proof}

%Let us remark that we do not require a closed structure on FwLSMCs since we are exclusively interested in FwLSMCs that are also BwLSMCCs (see Def.~\ref{DefBiLSMCCs}).

Naturally, it seems a reasonable idea to require the FwL-structure on NSCs to model LLK$^-$, but it is impossible for the game-semantic NSC $\mathcal{LG}$ in Sect.~\ref{GameSemantics}: The game-semantic $\invamp$ and $\wn$ are not well-defined on \emph{non-strict} \cite{amadio1998domains,mccusker1998games} strategies; they may generate \emph{concurrent} (or \emph{nondeterministic}) strategies from non-strict, sequential (or deterministic) strategies.
As we shall see, the non-strictness is caused by \emph{currying} of strategies, i.e., the closed and the FwL-structures of $\mathcal{LG}$ are incompatible. 

This suggests employing the lluf subBwLSMC $\sharp \mathcal{LG}$ of $\mathcal{LG}$ whose strategies are all \emph{strict}.
Of course, $\sharp \mathcal{LG}$ is \emph{not closed}, but currying \emph{w.r.t. the entire domain} and uncurrying \emph{w.r.t. the entire codomain} are possible. 
This observation actually motivates the rules $\Rightarrow$R$^-$ and $\multimap$R$^-$ given in Sect.~\ref{SequentCalculi}.
It also leads to:
\begin{lemma}[Semantics of LL$^-$]
\label{LemLLNegativeWithoutDst}
A NSC $\mathcal{C} = (\mathcal{C}, \otimes, \top, \multimap, \oc)$ equipped with a FwLSMC $\sharp \mathcal{C} = (\sharp \mathcal{C}, \invamp, \bot, \wn)$ such that:
\begin{enumerate}

\item $\sharp \mathcal{C}$ is a lluf subBwLSMC of $\mathcal{C}$, in which $\top$ is terminal, $\bot$ is initial, $\sharp\mathcal{C}(A, \bot) = \mathcal{C}(A, \bot)$ and $\sharp\mathcal{C}(\top, B) = \mathcal{C}(\top, B)$ for all $A, B \in \mathcal{C}$;

\item It is equipped with the following morphisms in $\sharp\mathcal{C}$:
\begin{align*}
%\label{NT1}
%\iota_{A, B} : A \otimes \bot &\rightarrow B \\
%\label{NT2}
\Omega_{A, B, C} : A \otimes (B \invamp C) &\rightarrow (A \otimes B) \invamp C \\
%\Phi_{A, B, C} : C \otimes (A \oplus B) &\rightarrow (C \otimes A) \oplus (C \otimes B) \\
%\Psi_{A, B, C} : (A \invamp C) \& (B \invamp C) &\rightarrow (A \& B) \invamp C \\
\Upsilon_A : \oc \wn A &\rightarrow \wn \oc A \\
%\phi_{A, B, C, D} : (A \invamp B) \otimes (C \invamp D) &\rightarrow (A \otimes C) \invamp B \invamp D \\
%\label{NT3}
%\psi_{A, B, C, D} : A \otimes B \otimes (C \invamp D) &\rightarrow (A\otimes C) \invamp (B \otimes D) \\
%\sigma_{A, B, C} : (A \invamp C) \& (B \invamp C) &\rightarrow (A \& B) \invamp C \\
%\tau_{A, B, C} : A \otimes (B \oplus C) &\rightarrow (A \otimes B) \oplus (A \otimes C) \\
\Sigma_{A, B} : \oc (A \invamp \wn B) &\rightarrow \oc A \invamp \wn B \\
\Pi_{A, B} : \oc A \otimes \wn B &\rightarrow \wn (\oc A \invamp B) 
\end{align*} 
natural in $A, B, C \in \mathcal{C}$;

\item It is equipped with the following isomorphisms in $\sharp\mathcal{C}$:
\begin{align}
%\label{NI1}
%A \otimes \bot &\cong \bot \\
\label{NI1}
A \invamp \top &\cong \top \\
\label{NI2}
A \multimap B &\cong \neg A \invamp B
\end{align}
natural in $A, B \in \mathcal{C}$, where $\neg A \stackrel{\mathrm{df. }}{=} A \multimap \bot$;

\end{enumerate}
gives a semantics of LLK$^-$ in $\sharp \mathcal{C}$.
\end{lemma}
\begin{proof}
We interpret proofs of each sequent $A_1, A_2, \dots, A_m \vdash B_1, B_2, \dots, B_n$ in LLK$^-$ by morphisms $A_1 \otimes A_2 \dots \otimes A_m \rightarrow B_1 \invamp B_2 \dots \invamp B_n$ in $\sharp \mathcal{C}$ by induction on the proofs, where we indicate the interpretation of logical constants and connectives of LL$^-$ by the notation for $\mathcal{C}$ (n.b., the domain of the morphisms is $\top$ if $m = 0$, and the codomain is $\bot$ if $n = 0$).

First, Id, $\top$R and $\bot$L are interpreted by identities in $\sharp \mathcal{C}$. 
We may handle $\top$L and $\bot$R by the unit laws of $\top$ and $\bot$, and $1$R and $0$L by (\ref{NI1}) and initiality of $0$, respectively. 
Also, the distribution rules are modeled by $\Omega$ and $\Upsilon$ in the obvious way.

We interpret Cut$^-$ by $\frac{ \ f : \Delta \rightarrow B \ \ \ \ f' : \Delta' \otimes B \rightarrow \Gamma' \ }{ \ \Delta \otimes \Delta' \stackrel{\varpi_{\Delta, \Delta'}}{\rightarrow} \Delta' \otimes \Delta \stackrel{\mathit{id}_{\Delta'} \otimes f}{\rightarrow} \Delta' \otimes B \stackrel{f'}{\rightarrow} \Gamma' \ }$, where $\varpi_{\Delta, \Delta'}$ is the symmetry w.r.t. $\otimes$.

The interpretations of $\otimes$L and $\invamp$R may be reduced to the induction hypotheses; $\otimes$R$^-$ and $\invamp$L$^-$ are interpreted by $\frac{ \ f_1 : \Delta_1 \rightarrow B_1 \ \ \ \ f_2 : \Delta_2 \rightarrow B_2 \ }{ \ f_1 \otimes f_2 : \Delta_1 \otimes \Delta_2 \rightarrow B_1 \otimes B_2 \ }$ and $\frac{ \ g_1 : A_1 \rightarrow \Gamma_1 \ \ \ \ g_2 : A_2 \rightarrow \Gamma_2 \ }{ \ g_1 \invamp g_2 : A_1 \invamp A_2 \rightarrow \Gamma_1 \invamp \Gamma_2 \ }$.
 
The interpretations of $\&$L and $\oplus$R are given by $\frac{ \ l_i : \Delta \otimes A_i \rightarrow \Gamma \ \ \ \ i \in \overline{2} \ }{ \ \Delta \otimes (A_1 \& A_2) \stackrel{\mathit{id}_\Delta \otimes \pi_i}{\rightarrow} \Delta \otimes A_i \stackrel{l_i}{\rightarrow} \Gamma \ }$ and $\frac{ \ r_i : \Delta \rightarrow B_i \invamp \Gamma \ \ \ \ i \in \overline{2} \ }{ \ \Delta \stackrel{r_i}{\rightarrow} B_i \invamp \Gamma \stackrel{\iota_i \invamp \mathit{id}_\Gamma}{\rightarrow} (B_1 \oplus B_2) \invamp \Gamma \ }$, and those of $\&$R$^-$ and $\oplus$L$^-$ by $\frac{ \ r_1 : \Delta \rightarrow B_1 \ \ \ \ r_2 : \Delta \rightarrow B_2 \ }{ \ \langle r_1, r_2 \rangle : \Delta \rightarrow B_1 \& B_2  \ }$ and by $\frac{ \ l_1 : A_1 \rightarrow \Gamma \ \ \ \ l_2 : A_2 \rightarrow \Gamma \ }{ \ [l_1, l_2] : A_1 \oplus A_2 \rightarrow \Gamma \ }$, respectively. 

The interpretation of $\multimap$L is given as follows. 
Given $A, B, C, D \in \mathcal{C}$, let $\Phi_{A, B, C, D} : A \otimes B \otimes (C \invamp D) \rightarrow (A \otimes C) \invamp (B \otimes D)$ be the natural transformation obtained by composing $\Omega$ and symmetries w.r.t. $\otimes$ in the obvious manner. 
Then, given $h_1 : \Delta_1 \rightarrow A \invamp \Gamma_1$ and $h_2 : \Delta_2 \otimes B \rightarrow \Gamma_2$ in $\sharp\mathcal{C}$, we compose $(A \invamp \Gamma_1) \otimes \neg A \cong (\neg A \otimes \top) \otimes (A \invamp \Gamma_1) \stackrel{\Phi_{A, \neg A, \top, \Gamma_1}}{\rightarrow} (\neg A \otimes A) \invamp (\top \otimes \Gamma_1) \cong (A \otimes \neg A) \invamp \Gamma_1$, for which we write $\Phi'_{A, \Gamma_1, \neg A}$.
Then, we obtain $\Delta_1 \otimes \Delta_2 \otimes (A \multimap B) \stackrel{\mathit{id}_{\Delta_1 \otimes \Delta_2} \otimes (\ref{NI2})}{\rightarrow} \Delta_1 \otimes \Delta_2 \otimes (\neg A \invamp B) \stackrel{\Phi_{\Delta_1, \Delta_2, \neg A, B}}{\rightarrow} (\Delta_1 \otimes \neg A) \invamp (\Delta_2 \otimes B) \stackrel{(h_1 \otimes \mathit{id}_{\neg A}) \invamp h_2}{\rightarrow} ((A \invamp \Gamma_1) \otimes \neg A) \invamp \Gamma_2 \stackrel{\Phi'_{A, \Gamma_1, \neg A} \invamp \mathit{id}_{\Gamma_2}}{\rightarrow} ((A \otimes \neg A) \invamp \Gamma_1) \invamp \Gamma_2 \stackrel{((\mathit{ev}_{A, \bot} \circ \varpi_{A, \neg A}) \invamp \mathit{id}_{\Gamma_1}) \invamp \mathit{id}_{\Gamma_2}}{\rightarrow} (\bot \invamp \Gamma_1) \invamp \Gamma_2 \cong \Gamma_1 \invamp \Gamma_2$, where $\mathit{ev}_{A, \bot} : \neg A \otimes A \rightarrow \bot$ is obtained from $\mathit{id}_{\neg A} : \neg A \rightarrow \neg A$ by uncurrying in $\mathcal{C}$ (n.b., $\mathit{ev}_{A, \bot}$ must be in $\sharp\mathcal{C}$ because its codomain is $\bot$).
The interpretation of $\multimap$R$^-$ is by currying in $\mathcal{C}$, for which $\sharp\mathcal{C}(\top, B) = \mathcal{C}(\top, B)$ for all $B \in \mathcal{C}$ is employed.

Note that $\oc$D, $\oc$W and $\oc$C may be handled just as in the interpretation of ILL in NSCs \cite{benton1992term,benton1993term}; $\wn$D, $\wn$W and $\wn$C are just symmetric. 
Also, XL and XR are interpreted by symmetries w.r.t. $\otimes$ and $\invamp$, respectively.

Finally, $\wn$L is interpreted by $\frac{ \ a : \oc \Delta \otimes A \rightarrow \wn \Gamma \ }{ \ \oc \Delta \otimes \wn A \stackrel{\Pi_{\Delta, A}}{\rightarrow} \wn (\oc \Delta \otimes A) \stackrel{\wn a}{\rightarrow} \wn \wn \Gamma \stackrel{\mu_\Gamma}{\rightarrow} \wn \Gamma \ }$, and $\oc$R by $\frac{ \ b : \oc \Delta \rightarrow B \invamp \wn \Gamma \ }{ \ \oc \Delta \stackrel{\delta_\Delta}{\rightarrow} \oc \oc \Delta \stackrel{\oc b}{\rightarrow} \oc (B \invamp \wn \Gamma) \stackrel{\Sigma_{B, \Gamma}}{\rightarrow} \oc B \invamp \wn \Gamma \ }$. 
\end{proof}

In particular, for each $A \in \mathcal{C}$, currying in $\mathcal{C}$ gives:
\begin{align*}
\sharp\mathcal{C}(A, A) &\cong \sharp\mathcal{C}(\top \otimes A, A) \\
&\cong \sharp\mathcal{C}(\top, A \multimap A) \\
&\cong \sharp\mathcal{C}(\top, \neg A \invamp A) \ \text{(by (\ref{NI2}))} 
\end{align*}
where $\neg A \stackrel{\mathrm{df. }}{=} A \multimap \bot$ is the \emph{\bfseries negation} of $A$.
This natural bijection allows $\sharp \mathcal{C}$ to model \emph{(linear) classical laws}:
\begin{itemize}

\item We may get:
\begin{equation*}
%\label{LEM}
\mathit{lem}_A \in \sharp\mathcal{C}(\top, \neg A \invamp A)
\end{equation*}
from $\mathit{id}_A \in \mathcal{C}(A, A)$, which models the classical \emph{law of excluded middle (LEM)} \cite{troelstra2000basic};

\item We may further compose: 
\begin{equation*}
%\label{DNE}
\mathit{dne}_A \in \sharp\mathcal{C}(\neg \neg A, A)
\end{equation*}
by $\neg \neg A \cong \neg \neg A \otimes \top \stackrel{\mathit{id}_{\neg \neg A} \otimes \mathit{lem}_A}{\rightarrow} \neg \neg A \otimes (\neg A \invamp A) \cong \neg \neg A \otimes \top \otimes (\neg A \invamp A) \stackrel{\Phi_{\neg \neg A, \top, \neg A, A}}{\rightarrow} (\neg \neg A \otimes \neg A) \invamp (\top \otimes A) \cong (\neg \neg A \otimes \neg A) \invamp A \stackrel{\mathit{ev}_{\neg A, \bot} \invamp \mathit{id}_A}{\rightarrow} \bot \invamp A \cong A$, which models \emph{double negation elimination (DNE)} \cite{troelstra2000basic}.

\end{itemize}

Recall that our aim is to give a \emph{unity} of logic; thus, we shall obtain semantics of CL from that of LL$^-$, i.e., NSCs satisfying the assumption of Lem.~\ref{LemLLNegativeWithoutDst}.
For this point, we employ:
\begin{definition}[Distributive laws \cite{power2002combining}]
Let $\mathcal{C}$ be a category, and $\wn = (\wn, \eta, \mu)$ and $\oc = (\oc, \epsilon, \delta)$ a monad and a comonad on $\mathcal{C}$. 
A \emph{\bfseries distributive law} of $\oc$ over $\wn$ is a natural transformation $d : \oc \wn \Rightarrow \wn \oc$ such that $\wn \epsilon \circ d = \epsilon \wn : \oc \wn \Rightarrow \wn$, $\wn \delta \circ d = d \oc \circ \oc d \circ \delta \wn : \oc \wn \Rightarrow \wn \oc \oc$, $d \circ \oc \eta = \eta \oc : \oc \Rightarrow \wn \oc$ and $d \circ \oc \mu = \mu \oc \circ \wn d \circ d \wn : \oc \wn \wn \Rightarrow \wn \oc$.
\end{definition}

\begin{theorem}[Bi-Kleisli extension \cite{power2002combining}]
Let $\wn$ and $\oc$ be a monad and a comonad on a category $\mathcal{C}$, and $d : \oc \wn \Rightarrow \wn \oc$ a distributive law of $\oc$ over $\wn$.
The Kleisli construction on $\mathcal{C}$ over $\wn$ is extended to the co-Kleisli category $\mathcal{C}_\oc$, and the co-Kleisli construction on $\mathcal{C}$ over $\oc$ to the Kleisli category $\mathcal{C}^\wn$.
Moreover, the extended Kleisli and co-Kleisli categories are equivalent, i.e., $(\mathcal{C}_\oc)^\wn  \simeq (\mathcal{C}^\wn)_\oc$.
\end{theorem}

Given a distributive law of $\oc$ over $\wn$, we define $\mathcal{C}_\oc^\wn \stackrel{\mathrm{df. }}{=} (\mathcal{C}_\oc)^\wn  \simeq (\mathcal{C}^\wn)_\oc$ and call it the \emph{\bfseries bi-Kleisli category} of $\mathcal{C}$ over $\oc$ and $\wn$.

As one may have already expected, we propose $\mathcal{C}_\oc^\wn$ as our categorical structure to interpret CL.
%Although we leave it as future work to develop a term calculus for CL and establish the soundness/completeness of the semantics of CL in $\mathcal{C}_G^F$ with respect to the calculus, the systematic construction of $\mathcal{C}_G^F$ indicates that it is a strong candidate. 
Hence, we define:
\begin{definition}[BiLSMCCs]
\label{DefBiLSMCCs}
A \emph{\bfseries bi-liberalizable (BiL) SMCC} is a NSC $\mathcal{C} = (\mathcal{C}, \otimes, \top, \multimap, \oc)$ such that it has weak finite coproducts $(0, \oplus)$, and it is equipped with:
\begin{itemize}

\item A lluf subBwLSMC $\sharp\mathcal{C}$ and a triple $(\invamp, \bot, \wn)$ such that $\top$ (resp. $\bot$) is terminal (resp. initial) in $\sharp\mathcal{C}$, $\sharp\mathcal{C}(A, \bot) = \mathcal{C}(A, \bot)$ and $\sharp\mathcal{C}(\top, B) = \mathcal{C}(\top, B)$ for all $A, B \in \mathcal{C}$, and $\sharp\mathcal{C} = (\sharp \mathcal{C}, \invamp, \bot, \wn)$ is FwL with finite coproducts $(0, \oplus)$;

\item Natural transformations $\Omega$, $\Sigma$ and $\Pi$ in $\sharp\mathcal{C}$ (Lem.~\ref{LemLLNegativeWithoutDst});

\item Natural isomorphisms (\ref{NI1}) and (\ref{NI2}) as well as 
\begin{align}
\label{NI3}
\wn (A \multimap B) &\cong \oc A \multimap \wn B \\
\label{NI4}
\neg A \oplus \neg B &\cong \neg (A \& B)
\end{align}
(natural in $A, B \in \mathcal{C}$) in $\sharp \mathcal{C}$;

\item A distributive law $\Upsilon$ of $\oc$ over $\wn$. 
\if0
such that the map
\begin{align}
\sharp \mathcal{C}(\oc A, \wn B) &\rightarrow \sharp \mathcal{C}(\oc \wn A, \wn \oc B) \nonumber \\
f &\mapsto d_A ; \wn \delta_A ; \wn \oc f ; \wn d_B ; \mu_{\oc B}  
\end{align}
has the inverse $g \mapsto \oc \eta_A ; g ; \wn \eta_B$ for all $A, B \in \mathcal{C}$.
\fi
\end{itemize}
\end{definition}

\begin{theorem}[Semantics of LL$^-$]
\label{ThmSemanticsOfLLNegative}
Each BiLSMCC $\mathcal{C}$ gives a semantics of LLK$^-$ in $\sharp\mathcal{C}$.
\end{theorem}
\begin{proof}
Immediate from Lem.~\ref{LemLLNegativeWithoutDst}.
\end{proof}

The natural isomorphisms (\ref{NI3}) and (\ref{NI4}) are not necessary for Thm.~\ref{ThmSemanticsOfLLNegative}, but they induce some of the De Morgan laws:
\begin{align}
\wn \neg A &= \wn (A \multimap \bot) \nonumber \\
&\cong \oc A \multimap \wn \bot \nonumber \ \text{(by (\ref{NI3}))} \\
&\cong \neg \oc A \invamp \wn \bot \nonumber \ \text{(by (\ref{NI2}))} \\
&\cong \neg \oc A \invamp \bot \nonumber \ \text{(by $\wn \bot \cong \wn 0 \cong \bot$)} \\
\label{DeMorgan?!}
&\cong \neg \oc A
\end{align}
natural in $A \in \mathcal{C}$, as well as:
\begin{align}
\neg A \invamp \neg B &\cong A \multimap (B \multimap \bot) \nonumber \ \text{(by (\ref{NI2}))} \\
&\cong (A \otimes B) \multimap \bot \nonumber \\
\label{DeMorganProdTensor}
&= \neg (A \otimes B)
\end{align}
natural in $A, B \in \mathcal{C}$.
Note that (\ref{NI4}) is also one of the De Morgan laws. 
As we shall see in Sect.~\ref{CutElimSoundnessAndCompleteness}, these natural isomorphisms exist in the game-semantic and the syntactic instances.

\if0
Similarly, the following laws are derivable in the bi-Kleisli category $\mathcal{C}_\oc^\wn$:
\begin{align}
\wn \neg A &= \wn (\oc A \multimap \wn \bot) \nonumber \\
&\cong \oc \oc A \multimap \wn \wn \bot \nonumber \ \text{(by (\ref{NI2}))} \\
&\cong \oc \oc A \multimap \wn \bot \nonumber \\
&\cong \neg \oc A
\end{align}
natural in $A \in \mathcal{C}$, 
\begin{align}
\neg A \vee \neg B &= \oc (\oc A \multimap \wn \bot) \oplus \oc (\oc B \multimap \wn \bot) \nonumber \\
&\cong
\end{align}
\fi

By Thm.~\ref{ThmTranslationOfCLIntoLL} and \ref{ThmSemanticsOfLLNegative}, a BiLSMCC $\mathcal{C}$ may interpret LK$^-$ in $\sharp\mathcal{C}$.
In addition, it is easy to see that the interpretation of LK$^-$ occurs always in the bi-Kleisli category $\sharp\mathcal{C}_\oc^\wn$, and therefore:
\begin{corollary}[Semantics of CL]
\label{CorSemanticsOfCL}
Let $\mathcal{C} = (\mathcal{C}, \otimes, \top, \multimap, \oc)$ together with $\sharp\mathcal{C} = (\sharp\mathcal{C}, \invamp, \bot, \wn, \Omega, \Sigma, \Pi, \Upsilon)$ be a BiLSMCC.
The bi-Kleisli category $\sharp\mathcal{C}_\oc^\wn$ gives a semantics of LK$^-$ such that $\wedge$, $\vee$ and $\Rightarrow$ are interpreted by: 
\begin{align*}
A \wedge B &\stackrel{\mathrm{df. }}{=} \wn A \& \wn B \\
A \vee B &\stackrel{\mathrm{df. }}{=} \oc A \oplus \oc B \\
A \Rightarrow B &\stackrel{\mathrm{df. }}{=} \oc \wn A \multimap \wn \oc B
\end{align*}
for all $A, B \in \mathcal{C}$, and there are natural isomorphisms in $\sharp\mathcal{C}$:
\begin{align*}
\neg (A \wedge B) &\cong \neg A \vee \neg B \\
A \Rightarrow B &\cong \wn (\neg A \vee B)
\end{align*}
and a natural transformation in $\sharp\mathcal{C}$:
\begin{equation*}
\neg \neg A \rightarrow A.
\end{equation*}
\end{corollary}
\begin{proof}
By the proofs of Thm.~\ref{ThmTranslationOfCLIntoLL} and \ref{ThmSemanticsOfLLNegative}, the interpretation of LK$^-$ in $\sharp\mathcal{C}$ actually occurs in $\sharp\mathcal{C}_\oc^\wn$; thus, it suffices to establish the natural isomorphisms and transformation. 
Then, we have $\neg (A \wedge B) = \neg (\wn A \& \wn B) \stackrel{(\ref{NI4})}{\cong} \neg \wn A \oplus \neg \wn B \stackrel{(\ref{DeMorgan?!})}{\cong} \oc \neg A \oplus \oc \neg B = \neg A \vee \neg B$ and $A \Rightarrow B = \oc \wn A \multimap \wn \oc B \stackrel{(\ref{NI3})}{\cong} \wn (\wn A \multimap \oc B) \stackrel{(\ref{NI2})}{\cong} \wn (\neg \wn A \invamp \oc B) \stackrel{(\ref{DeMorgan?!})}{\cong} \wn (\oc \neg A \invamp \oc B) = \wn (\neg A \vee B)$.
Finally, we have $\oc (\neg \neg A) \stackrel{\epsilon_A}{\rightarrow} \neg \neg A \stackrel{\mathit{dne}_A}{\rightarrow} A \stackrel{\eta_A}{\rightarrow} \wn A$, completing the proof.
\end{proof}

Note that the interpretation of Cor.~\ref{CorSemanticsOfCL} matches the translation $\mathscr{T}_{\mathsf{c}}$.
On the other hand, $\neg (A \vee B)$ and $\neg A \wedge \neg B$ (resp. $\neg (A \invamp B)$ and $\neg A \otimes \neg B$, $\neg (A \oplus B)$ and $\neg A \& \neg B$) should \emph{not} be isomorphic for they are not in the game semantics below.

\section{Game Semantics}
\label{GameSemantics}
This section gives a game-semantic BiLSMCC $\mathcal{LG}$. %and shows that the semantics of LLK$^-$, LK$^-$, LLJ and LJ in $\sharp\mathcal{LG}$, $\sharp\mathcal{LG}_\oc^\wn$, $\mathcal{LG}^{\mathsf{wc}}$ and $\mathcal{LG}^{\mathsf{wc}}_\oc$ are all fully complete.
We employ Guy McCusker's games and strategies \cite{mccusker1998games}. 

%We shall define, on the SMCC $\mathcal{G}$ of McCusker's games and strategies, new concepts and structures, so that it is extended to be a BiLSMCC $\mathcal{LG}$, and establish fully complete game semantics of LLK$^-$, LLJ, LK$^-$ and LJ$^+$.

\subsection{Review: Game-Semantic NSC}
\if0
\begin{convention}
Games $A$ and $B$ are written $A \simeq_{\mathcal{G}} B$ (n.b., written $A = B$ if they coincide \emph{on the nose}) if they are the same game up to `tags' on moves for disjoint union. 
Clearly, $A \simeq_{\mathcal{G}} B$ implies $A \cong B$ in $\mathcal{G}$, but not vice versa.
\end{convention}
\fi
%Let us quickly review McCusker's games and strategies, leaving more details and explanations to \cite{mccusker1998games,abramsky1999game}.

\begin{notation}
Given a finite sequence $\boldsymbol{s} = x_1 x_2 \dots x_{|\boldsymbol{s}|}$, where $|\boldsymbol{s}|$ is the \emph{length} of $\boldsymbol{s}$, we write $\boldsymbol{s}(i)$ for $x_i$ ($i \in \overline{|\boldsymbol{s}|}$).
We define $\mathsf{Even}(\boldsymbol{s}) \stackrel{\mathrm{df. }}{\Leftrightarrow} |\boldsymbol{s}| \equiv_2 0$, where $\equiv_2$ is the equality on $\mathbb{N}$ modulo $2$, and $\mathsf{Odd}(\boldsymbol{s}) \stackrel{\mathrm{df. }}{\Leftrightarrow} |\boldsymbol{s}| \equiv_2 1$ for a finite sequence $\boldsymbol{s}$, and $S^\mathsf{Even} \stackrel{\mathrm{df. }}{=} \{ \boldsymbol{s} \in S \mid \mathsf{Even}(\boldsymbol{s}) \ \! \}$ and $S^{\mathsf{Odd}} \stackrel{\mathrm{df. }}{=} S \setminus S^{\mathsf{Even}}$ for a set $S$ of finite sequences.
We write $\boldsymbol{\epsilon}$ for the \emph{empty sequence}.
\end{notation}

Recall that games are based on \emph{arenas} and \emph{legal positions}: An arena defines the basic components of a game, which in turn induces its legal positions that specify the basic rules of the game. 
Let us first recall these two preliminary concepts.

\begin{definition}[Arenas \cite{mccusker1998games}]
\label{DefArenas}
An \emph{\bfseries arena} is a triple $G = (M_G, \lambda_G, \vdash_G)$, where:
\begin{itemize}
\item $M_G$ is a set whose elements are called \emph{\bfseries moves};

\item $\lambda_G$ is a function from $M_G $ to $\{ \mathsf{O}, \mathsf{P} \} \times \{\mathsf{Q}, \mathsf{A} \}$, called the \emph{\bfseries labeling function}, in which $\mathsf{O}$, $\mathsf{P}$, $\mathsf{Q}$ and $\mathsf{A}$ are arbitrarily fixed symbols, called the \emph{\bfseries labels};

\item $\vdash_G$ is a subset of $(\{ \star \} \cup M_G) \times M_G$, where $\star$ is an arbitrarily fixed element such that $\star \not \in M_G$, called the \emph{\bfseries enabling relation}, that satisfies:
\begin{itemize}

\item \textsc{(E1)} $\star \vdash_G m$ implies $\lambda_G (m) = \mathsf{OQ} \wedge (n \vdash_G m \Leftrightarrow n = \star)$;

\item \textsc{(E2)} $m \vdash_G n \wedge \lambda_G^\mathsf{QA} (n) = \mathsf{A}$ implies $\lambda_G^\mathsf{QA} (m) = \mathsf{Q}$;

\item \textsc{(E3)} $m \vdash_G n \wedge m \neq \star$ implies $\lambda_G^\mathsf{OP} (m) \neq \lambda_G^\mathsf{OP} (n)$

\end{itemize}
in which $\lambda_G^\mathsf{OP} \stackrel{\mathrm{df. }}{=} \pi_1 \circ \lambda_G : M_G \to \{ \mathsf{O}, \mathsf{P} \}$ and $\lambda_G^\mathsf{QA} \stackrel{\mathrm{df. }}{=} \pi_2 \circ \lambda_G : M_G \to \{ \mathsf{Q}, \mathsf{A} \}$.

\end{itemize}
A move $m \in M_G$ is \emph{\bfseries initial} if $\star \vdash_G m$, an \emph{\bfseries O-move} (resp. a \emph{\bfseries P-move}) if $\lambda_G^\mathsf{OP}(m) = \mathsf{O}$ (resp. if $\lambda_G^{\mathsf{OP}}(m) = \mathsf{P}$), a \emph{\bfseries question} (resp. an \emph{\bfseries answer}) if $\lambda_G^\mathsf{QA}(m) = \mathsf{Q}$ (resp. if $\lambda_G^\mathsf{QA}(m) = \mathsf{A}$).
Let $M_G^{\mathsf{Init}} \stackrel{\mathrm{df. }}{=} \{ m \in M_G \mid \star \vdash_G m \ \! \}$ and $M_G^{\mathsf{nInit}} \stackrel{\mathrm{df. }}{=} M_G \setminus M_G^{\mathsf{Init}}$.
\end{definition}

\begin{definition}[Occurrences of moves]
\label{DefMoveOccurrences}
Given a finite sequence $\boldsymbol{s} \in M_G^\ast$ of moves of an arena $G$, an \emph{\bfseries occurrence (of a move)} in $\boldsymbol{s}$ is a pair $(\boldsymbol{s}(i), i)$ such that $i \in \overline{|\boldsymbol{s}|}$. 
More specifically, we call the pair $(\boldsymbol{s}(i), i)$ an \emph{\bfseries initial occurrence} (resp. a \emph{\bfseries non-initial occurrence}) in $\boldsymbol{s}$ if $\star \vdash_G \boldsymbol{s}(i)$ (resp. otherwise).
\end{definition}

To be exact, positions of games are not finite sequences but: %it is equipped with a \emph{justification} (or \emph{pointer}) structure:
\begin{definition}[J-sequences \cite{mccusker1998games}]
\label{DefJSequences}
A \emph{\bfseries justified (j-) sequence} of an arena $G$ is a pair $\boldsymbol{s} = (\boldsymbol{s}, \mathcal{J}_{\boldsymbol{s}})$ of a finite sequence $\boldsymbol{s} \in M_G^\ast$ and a map $\mathcal{J}_{\boldsymbol{s}} : \overline{|\boldsymbol{s}|} \rightarrow \{ 0 \} \cup \overline{|\boldsymbol{s}|-1}$ such that for all $i \in \overline{|\boldsymbol{s}|}$ $\mathcal{J}_{\boldsymbol{s}}(i) = 0$ if $\star \vdash_G \boldsymbol{s}(i)$, and $0 < \mathcal{J}_{\boldsymbol{s}}(i) < i \wedge \boldsymbol{s}({\mathcal{J}_{\boldsymbol{s}}(i)}) \vdash_G \boldsymbol{s}(i)$ otherwise.
The occurrence $(\boldsymbol{s}({\mathcal{J}_{\boldsymbol{s}}(i)}), \mathcal{J}_{\boldsymbol{s}}(i))$ is called the \emph{\bfseries justifier} of a non-initial occurrence $(\boldsymbol{s}(i), i)$ in $\boldsymbol{s}$.
%We also say that $(\boldsymbol{s}(i), i)$ is {\bfseries justified} by $(\boldsymbol{s}({\mathcal{J}_{\boldsymbol{s}}(i)}), \mathcal{J}_{\boldsymbol{s}}(i))$, or there is a {\bfseries pointer} from the former to the latter.
\end{definition}

\begin{notation}
We write $\mathscr{J}_G$ for the set of all j-sequences of an arena $G$, and $\boldsymbol{s} = \boldsymbol{t}$ for any $\boldsymbol{s}, \boldsymbol{t} \in \mathscr{J}_G$ if $\boldsymbol{s} = \boldsymbol{t}$ and $\mathcal{J}_{\boldsymbol{s}} = \mathcal{J}_{\boldsymbol{t}}$.
\end{notation}

The idea is that each non-initial occurrence in a j-sequence must be performed for a specific previous occurrence, viz., its justifier, in the j-sequence. 

\begin{remark}
Henceforth, by abuse of notation, we keep the pointer structure $\mathcal{J}_{\boldsymbol{s}}$ of each j-sequence $\boldsymbol{s} = (\boldsymbol{s}, \mathcal{J}_{\boldsymbol{s}})$ implicit and abbreviate occurrences $(\boldsymbol{s}(i), i)$ in $\boldsymbol{s}$ as $\boldsymbol{s}(i)$.
Moreover, we usually write $\mathcal{J}_{\boldsymbol{s}}(\boldsymbol{s}(i)) = \boldsymbol{s}(j)$ if $\mathcal{J}_{\boldsymbol{s}}(i) = j$.
%A \emph{\bfseries QA-pair} in a j-sequence $\boldsymbol{s}$ is a pair $(q, a)$ of a question $q$ and an answer $a$ justified by $q$ in $\boldsymbol{s}$.
%This convention is mathematically imprecise, but it does not bring any serious confusion in practice. 
\end{remark}

\begin{definition}[J-subsequences]
\label{DefJSubsequences}
Let $G$ be an arena, and $\boldsymbol{s} \in \mathscr{J}_G$. 
A \emph{\bfseries j-subsequence} of $\boldsymbol{s}$ is any $\boldsymbol{t} \in \mathscr{J}_G$ that satisfies:
\begin{itemize}

\item $\boldsymbol{t}$ is a subsequence of $\boldsymbol{s}$, written $(\boldsymbol{s}(i_1), \boldsymbol{s}(i_2), \dots, \boldsymbol{s}(i_{|\boldsymbol{t}|}))$; 

\item $\mathcal{J}_{\boldsymbol{t}}(\boldsymbol{s}(i_r)) = \boldsymbol{s}(i_l)$ iff there are occurrences $\boldsymbol{s}(j_1), \boldsymbol{s}(j_2), \dots, \boldsymbol{s}(j_{k})$ in $\boldsymbol{s}$ eliminated in $\boldsymbol{t}$ such that $\mathcal{J}_{\boldsymbol{s}}(\boldsymbol{s}(i_r)) = \boldsymbol{s}(j_1) \wedge \mathcal{J}_{\boldsymbol{s}}(\boldsymbol{s}(j_1)) = \boldsymbol{s}(j_2) \dots \wedge \mathcal{J}_{\boldsymbol{s}}(\boldsymbol{s}(j_{k-1})) = \boldsymbol{s}(j_{k}) \wedge \mathcal{J}_{\boldsymbol{s}}(\boldsymbol{s}(j_{k})) = \boldsymbol{s}(i_l)$.

\end{itemize}
\end{definition}

Next, let us recall `relevant part' of previous occurrences:
\begin{definition}[Views \cite{mccusker1998games}] 
The \emph{\bfseries Player (P-) view} $\lceil \boldsymbol{s} \rceil_G$ of a j-sequence $\boldsymbol{s} \in \mathscr{J}_G$ of an arena $G$ is the j-subsequences of $\boldsymbol{s}$ given by the following induction on $|\boldsymbol{s}|$: $\lceil \boldsymbol{\epsilon} \rceil_G \stackrel{\mathrm{df. }}{=} \boldsymbol{\epsilon}$; $\lceil \boldsymbol{s} m \rceil_G \stackrel{\mathrm{df. }}{=} \lceil \boldsymbol{s} \rceil_G . m$ if $m$ is a P-move; $\lceil \boldsymbol{s} m \rceil_G \stackrel{\mathrm{df. }}{=} m$ if $m$ is initial; and  $\lceil \boldsymbol{s} m \boldsymbol{t} n \rceil_G \stackrel{\mathrm{df. }}{=} \lceil \boldsymbol{s} \rceil_G . m n$ if $n$ is an O-move such that $m$ justifies $n$. %where the justifiers of occurrences in $\lceil \boldsymbol{s} \rceil_G$ (resp. $\lfloor \boldsymbol{s} \rfloor_G$) are unchanged if they occur in $\lceil \boldsymbol{s} \rceil_G$ (resp. $\lfloor \boldsymbol{s} \rfloor_G$), and undefined otherwise.
The \emph{\bfseries Opponent (O-) view} $\lfloor \boldsymbol{s} \rfloor_G$ of $\boldsymbol{s}$ is symmetric to $\lceil \boldsymbol{s} \rceil_G$.
%A {\bfseries view} is a P- or O-view.
\end{definition}

\if0
The idea behind the notion of views is as follows.
Given a j-sequence $\boldsymbol{s} m$ of an arena $G$ such that $m$ is a P-move (resp. an O-move), the P-view $\lceil \boldsymbol{s} \rceil$ (resp. the O-view $\lfloor \boldsymbol{s} \rfloor$) is intended to be the currently `relevant' part of previous occurrences for Player (resp. Opponent). That is, Player (resp. Opponent) is concerned only with the last occurrence of an O-move (resp. a P-move), its justifier and that justifier's `concern', i.e., P-view (resp. O-view), which then recursively proceeds.
See \cite{hyland2000full,curien2006notes,curien1998abstract} for an explanation of justifiers and views in terms of their correspondences with syntax.
\fi

\if0
\begin{remark}
A P- or O-view may not be a j-sequence, motivating \emph{visibility} below.
\end{remark}
\fi

We may now recall \emph{legal positions} of an arena \cite{mccusker1998games,abramsky1999game}:
\begin{definition}[Legal positions \cite{mccusker1998games,abramsky1999game}]
\label{DefLegalPositions}
A \emph{\bfseries legal position} of an arena $G$ is a j-sequence $\boldsymbol{s} \in \mathscr{J}_G$ that satisfies:
\begin{itemize}

\item \textsc{(Alternation)} $\boldsymbol{s} = \boldsymbol{s_1} m n \boldsymbol{s_2} \Rightarrow \lambda_G^\mathsf{OP} (m) \neq \lambda_G^\mathsf{OP} (n)$;

%\item \mathsf{\bfseries Bracketing.} If $\boldsymbol{t} q \boldsymbol{u} a \preceq \boldsymbol{s}$, where the question $q$ is ``answered'' by the answer $a$ (i.e., $q$ justifies $a$), then there is no ``unanswered'' question in $\boldsymbol{u}$.

\item \textsc{(Visibility)} If $\boldsymbol{s} = \boldsymbol{t} m \boldsymbol{u}$ with $m$ non-initial, then $\mathcal{J}_{\boldsymbol{s}}(m)$ occurs in $\lceil \boldsymbol{t} \rceil_{G}$ if $m$ is a P-move, and in $\lfloor \boldsymbol{t} \rfloor_{G}$ otherwise.

\end{itemize}
\end{definition}

\begin{notation}
We write $\mathscr{L}_G$ for the set of all legal positions of $G$. 
\end{notation}

We are now ready to recall the following central notion: 
\begin{definition}[Games \cite{mccusker1998games,abramsky1999game}]
\label{DefGames}
A \emph{\bfseries game} is a quintuple $G = (M_G, \lambda_G, \vdash_G, P_G, \simeq_G)$ such that $(M_G, \lambda_G, \vdash_G)$ is an arena, $P_G$ is a non-empty, prefix-closed subset of $\mathscr{L}_G$, whose elements are called \emph{\bfseries (valid) positions} of $G$, and $\simeq_G$ is an equivalence relation on $P_G$, called the \emph{\bfseries identification of (valid) positions} of $G$, that satisfies:

\begin{itemize}

\item \textsc{(I1)} $\boldsymbol{s} \simeq_G \boldsymbol{t} \Rightarrow |\boldsymbol{s}| = |\boldsymbol{t}|$;

\item \textsc{(I2)} $\boldsymbol{s} m \simeq_G \boldsymbol{t} n \Rightarrow \boldsymbol{s} \simeq_G \boldsymbol{t} \wedge \lambda_G(m) = \lambda_G(n) \wedge (m, n \in M_G^{\mathsf{Init}} \vee (\exists i \in \overline{|\boldsymbol{s}|} . \ \! \mathcal{J}_{\boldsymbol{s} m}(m) = \boldsymbol{s}(i) \wedge \mathcal{J}_{\boldsymbol{t} n}(n) = \boldsymbol{t}(i)))$;

\item \textsc{(I3)} $\boldsymbol{s} \simeq_G \boldsymbol{t} \wedge \boldsymbol{s} m \in P_G \Rightarrow \exists \boldsymbol{t} n \in P_G . \ \! \boldsymbol{s} m \simeq_G \boldsymbol{t} n$.

\end{itemize}
%It is {\bfseries well-founded (wf-)} if so is $\vdash_G$, i.e., there is no infinite sequence $\star \vdash_G m_1 \vdash_G m_2 \vdash_G \dots$.
%A {\bfseries play} of $G$ is a (finite or infinite) sequence $\boldsymbol{\epsilon}, m_1, m_1m_2, \dots$ of positions in $G$.
\end{definition}

The set $P_G$ is non-empty because there is always the starting position or `moment' of a game $G$, and prefix-closed because each non-empty `moment' of $G$ must have the previous `moment'.
Identifications of positions are originally introduced in \cite{abramsky2000full} and also employed in Section 3.6 of \cite{mccusker1998games}.
They are to identify positions up to inessential details of `tags' for disjoint union of sets of moves for \emph{exponential} $\oc$ (Def.~\ref{DefExponential}).
For this underlying idea, the axioms I1, I2 and I3 should make sense.

Recall that a game $G$ is \emph{\bfseries well-founded (wf)} if so is $\vdash_G$ \cite{clairambault2010totality}, i.e., there is no infinite sequence $\star \vdash m_1 \vdash m_2 \vdash m_3 \dots$, and \emph{\bfseries well-opened (wo)} if $\boldsymbol{s} m \in P_G \wedge m \in M_G^{\mathsf{Init}} \Rightarrow \boldsymbol{s} = \boldsymbol{\epsilon}$ \cite{mccusker1998games}.

\if0
In addition, let us introduce:
\begin{definition}[Well-closure of games]
A game $G$ is \emph{\bfseries well-closed (wc)} if $\boldsymbol{s} = \boldsymbol{t} q \boldsymbol{u} a \in P_G$ such that whenever $q$ is an initial question, and $a$ is an answer justified by $q$ in $\boldsymbol{s}$, there is no answer justified by $q$ in $\boldsymbol{u}$. 
\end{definition}
\fi

The \emph{\bfseries top game} $\top \stackrel{\mathrm{df. }}{=} (\emptyset, \emptyset, \emptyset, \{ \boldsymbol{\epsilon} \}, \{ (\boldsymbol{\epsilon}, \boldsymbol{\epsilon}) \})$ and the \emph{\bfseries bottom game} $\bot \stackrel{\mathrm{df. }}{=} (\{ q \}, q \mapsto \mathsf{OQ}, \{ (\star, q) \}, \{ \boldsymbol{\epsilon}, q \}, \{ (\boldsymbol{\epsilon}, \boldsymbol{\epsilon}), (q, q) \})$ are, e.g., wf and wo.
We also write $1$ and $0$ for $\top$ and $\bot$, and call them the \emph{\bfseries one game} and the \emph{\bfseries zero game}, respectively.
%We leave constructions on games such as \emph{\bfseries tensor} $\otimes$, \emph{\bfseries affine implication} $\rightarrowtriangle$\footnote{It is usually called \emph{\bfseries linear implication} and written $\multimap$ \cite{mccusker1998games,abramsky1999game}.}, \emph{\bfseries product} $\&$ and \emph{\bfseries exponential} $\oc$ to \cite{mccusker1998games}.

Now, let us recall standard constructions on games.
For brevity, we usually omit `tags' for disjoint union of sets. %keeping in mind that formally tags work as described just above Definition~\ref{DefEqualityBetweenGames}. 
For instance, we write $x \in A + B$ iff $x \in A$ or $x \in B$; also, given relations $R_A \subseteq A \times A$ and $R_B \subseteq B \times B$, we write $R_A + R_B$ for the relation on $A + B$ such that $(x, y) \in R_A + R_B \stackrel{\mathrm{df. }}{\Leftrightarrow} (x, y) \in R_A \vee (x, y) \in R_B$. 

We first review \emph{tensor} $\otimes$.
A position of the tensor $A \otimes B$ of games $A$ and $B$ is an interleaving mixture of positions of $A$ and $B$, in which only Opponent may switch the $AB$-parity. 
\begin{definition}[Tensor product of games \cite{mccusker1998games,abramsky1999game}]
\label{DefTensorOfGames}
The \emph{\bfseries tensor (product)} $A \otimes B$ of games $A$ and $B$ is defined by:
\begin{itemize}

\item $M_{A \otimes B} \stackrel{\mathrm{df. }}{=} M_A + M_B$;

\item $\lambda_{A \otimes B} \stackrel{\mathrm{df. }}{=} [\lambda_A, \lambda_B]$;

\item $\vdash_{A \otimes B} \ \stackrel{\mathrm{df. }}{=} \ \vdash_A + \ \! \vdash_B$;

\item $P_{A \otimes B} \stackrel{\mathrm{df. }}{=} \{ \boldsymbol{s} \in \mathscr{L}_{A \otimes B} \mid \boldsymbol{s} \upharpoonright A \in P_A, \boldsymbol{s} \upharpoonright B \in P_B \ \! \}$, where $\boldsymbol{s} \upharpoonright A$ (resp. $\boldsymbol{s} \upharpoonright B$) is the j-subsequence of $\boldsymbol{s}$ that consists of moves of $A$ (resp. $B$);

\item $\boldsymbol{s} \simeq_{A \otimes B} \boldsymbol{t} \stackrel{\mathrm{df. }}{\Leftrightarrow} \boldsymbol{s} \upharpoonright A \simeq_A \boldsymbol{t} \upharpoonright A \wedge \boldsymbol{s} \upharpoonright B \simeq_B \boldsymbol{t} \upharpoonright B \wedge \mathit{att}_{A \otimes B}^\ast(\boldsymbol{s}) = \mathit{att}_{A \otimes B}^\ast(\boldsymbol{t})$, where $\mathit{att}_{A \otimes B} : M_{A \otimes B} \rightarrow \{ 0, 1 \}$ maps $a \in M_A \mapsto 0, b \in M_B \mapsto 1$.
\end{itemize}
\end{definition}

It is easy to see that in fact only Opponent may switch the $AB$-parity of moves during a play of $A \otimes B$ by alternation.

Next, let us recall the space of \emph{linear functions} \cite{girard1987linear,girard2011blind}: 
\begin{definition}[Linear implication between games \cite{mccusker1998games,abramsky1999game}]
\label{DefLinearImplication}
The \emph{\bfseries linear implication} $A \multimap B$ from a game $A$ to another $B$ is defined by:
\begin{itemize}

\item $M_{A \multimap B} \stackrel{\mathrm{df. }}{=} M_A + M_B$;

\item $\lambda_{A \multimap B} \stackrel{\mathrm{df. }}{=} [\overline{\lambda_A}, \lambda_B]$, where $\overline{\lambda_A} \stackrel{\mathrm{df. }}{=} \langle \overline{\lambda_A^\mathsf{OP}}, \lambda_A^\mathsf{QA} \rangle$ and $\overline{\lambda_A^\mathsf{OP}} (m) \stackrel{\mathrm{df. }}{=} \begin{cases} \mathsf{P} \ &\text{if $\lambda_A^\mathsf{OP} (m) = \mathsf{O}$} \\ \mathsf{O} \ &\text{otherwise} \end{cases}$;

\item $\vdash_{A \multimap B} \stackrel{\mathrm{df. }}{=} \{ (\star, \hat{b}) \mid \star \vdash_B \hat{b} \ \! \} + \{ (\hat{b}, \hat{a}) \mid \star \vdash_A \hat{a}, \star \vdash_B \hat{b} \ \! \} \\ + (\vdash_A \cap \ (M_A \times M_A)) + (\vdash_B \cap \ (M_B \times M_B))$;

\item $P_{A \multimap B} \stackrel{\mathrm{df. }}{=} \{ \boldsymbol{s} \in \mathscr{L}_{A \multimap B} \mid \boldsymbol{s} \upharpoonright A \in P_A, \boldsymbol{s} \upharpoonright B \in P_B \ \! \}$;

\item $\boldsymbol{s} \simeq_{A \multimap B} \boldsymbol{t} \stackrel{\mathrm{df. }}{\Leftrightarrow} \boldsymbol{s} \upharpoonright A \simeq_A \boldsymbol{t} \upharpoonright A \wedge \boldsymbol{s} \upharpoonright B \simeq_B \boldsymbol{t} \upharpoonright B \wedge \mathit{att}_{A \multimap B}^\ast(\boldsymbol{s}) = \mathit{att}_{A \multimap B}^\ast(\boldsymbol{t})$, where $\mathit{att}_{A \multimap B} : M_{A \multimap B} \rightarrow \{ 0, 1 \}$ maps $a \in M_A \mapsto 0, b \in M_B \mapsto 1$.

\end{itemize}
\end{definition}

Similarly to tensor $A \otimes B$, a position of the linear implication $A \multimap B$ is an interleaving mixture of positions of $A$ and $B$, but only Player may switch the $AB$-parity again by alternation.  

For lack of space, we leave the details of \emph{\bfseries product} $\&$ on games to \cite{mccusker1998games}.
Roughly, the set $P_{A \& B}$ of all positions of the product $A \& B$ of games $A$ and $B$ is the disjoint union $P_A + P_B$.

\if0
The following is a product in the category of games \cite{mccusker1998games}:
\begin{definition}[Product of games \cite{abramsky1994games,abramsky1999game}]
\label{DefProduct}
The \emph{\bfseries product} $A \& B$ of games $A$ and $B$ is defined by:
\begin{itemize}

\item $M_{A \& B} \stackrel{\mathrm{df. }}{=} M_A + M_B$;

\item $\lambda_{A \& B} \stackrel{\mathrm{df. }}{=} [\lambda_A, \lambda_B]$;

\item $\vdash_{A \& B} \ \stackrel{\mathrm{df. }}{=} \ \vdash_A + \vdash_B$;

\item $P_{A \& B} \stackrel{\mathrm{df. }}{=} \{ \boldsymbol{s} \in \mathscr{L}_{A \& B} \mid (\boldsymbol{s} \upharpoonright A \in P_A \wedge \boldsymbol{s} \upharpoonright B = \boldsymbol{\epsilon}) \\ \vee (\boldsymbol{s} \upharpoonright A = \boldsymbol{\epsilon} \wedge \boldsymbol{s} \upharpoonright B \in P_B) \ \! \}$;

\item $\boldsymbol{s} \simeq_{A \& B} \boldsymbol{t} \stackrel{\mathrm{df. }}{\Leftrightarrow} \boldsymbol{s} \simeq_A \boldsymbol{t} \vee \boldsymbol{s} \simeq_B \boldsymbol{t}$.

\end{itemize}
\end{definition}
\fi

Next, we introduce \emph{coproduct} or \emph{sum} of games:
\begin{definition}[Sum of games]
\label{DefSum}
The \emph{\bfseries sum} $A \oplus B$ of games $A$ and $B$ is defined by:
\begin{itemize}

\item $M_{A \oplus B} \stackrel{\mathrm{df. }}{=} (M_A^{\mathsf{Init}} \times M_B^{\mathsf{Init}}) + M_A + M_B$;

\item $\lambda_{A \oplus B} : (\hat{a}, \hat{b}) \in M_A^{\mathsf{Init}} \times M_B^{\mathsf{Init}} \mapsto \mathsf{OQ}, a \in M_A \mapsto \lambda_A(a), b \in M_B \mapsto \lambda_B(b)$;

\item $\vdash_{A \oplus B} \ \stackrel{\mathrm{df. }}{=} \{ (\star, (\hat{a}, \hat{b})) \mid \star \vdash_A \hat{a}, \star \vdash_B \hat{b} \ \! \} + \vdash_A + \vdash_B \\ + \{ ((\hat{a}, \hat{b}), a) \in (M_A^{\mathsf{Init}} \times M_B^{\mathsf{Init}}) \times M_A \mid \hat{a} \vdash_A a \ \! \} \\ + \{ ((\hat{a}, \hat{b}), b) \in (M_A^{\mathsf{Init}} \times M_B^{\mathsf{Init}}) \times M_B \mid \hat{b} \vdash_B b \ \! \}$;

\item $P_{A \oplus B} \stackrel{\mathrm{df. }}{=} \{ \boldsymbol{s} \in \mathscr{L}_{A \& B} \mid \boldsymbol{s} \upharpoonright A \in P_A \vee \boldsymbol{s} \upharpoonright B \in P_B, \boldsymbol{s} = x \boldsymbol{t} \Rightarrow x \in (M_A^{\mathsf{Init}} \times M_B^{\mathsf{Init}}) \cap (P_A \times P_B) \}$, where $\boldsymbol{s} \upharpoonright A$ (resp. $\boldsymbol{s} \upharpoonright B$) is the j-subsequence of $\boldsymbol{s}$ that consists of moves $(\hat{a}, \hat{b}) \in M_A^{\mathsf{Init}} \times M_B^{\mathsf{Init}}$ and $a \in M_A$ (resp. $b \in M_B$) with the former changed into $\hat{a}$ (resp. $\hat{b}$);

\item $\boldsymbol{s} \simeq_{A \oplus B} \boldsymbol{t} \stackrel{\mathrm{df. }}{\Leftrightarrow} \boldsymbol{s} \upharpoonright A \simeq_A \boldsymbol{t} \upharpoonright A \vee \boldsymbol{s} \upharpoonright B \simeq_B \boldsymbol{t} \upharpoonright B$.

\end{itemize}
\end{definition}

That is, a non-empty position of $A \oplus B$ is of the form $(\hat{a}, \hat{b}) \boldsymbol{t}$ such that $\hat{a} \boldsymbol{t} \in P_A \vee \hat{b} \boldsymbol{t} \in P_B$.
It is easy to see that an initial move of the form $(\hat{a}, \hat{b}) \in M_A^{\mathsf{Init}} \times M_B^{\mathsf{Init}}$ may occur in a position $\boldsymbol{s}$ only as the first element of $\boldsymbol{s}$. 
Our sum of games is different from the one given in \cite{mccusker1998games} to give a \emph{unity} of logic.

Now, let us recall the game semantics of \emph{of course} $\oc$: 
\begin{definition}[Exponential of games \cite{mccusker1998games}]
\label{DefExponential}
The \emph{\bfseries exponential} (or \emph{\bfseries of course}) $\oc A$ of a game $A$ is defined by:
\begin{itemize}

\item $M_{\oc A} \stackrel{\mathrm{df. }}{=} M_A \times \mathbb{N}$;

\item $\lambda_{\oc A} : (a, i) \in M_A \times \mathbb{N} \mapsto \lambda_A(a)$;

\item $\vdash_{\oc A} \stackrel{\mathrm{df. }}{=} \{ (\star, (\hat{a}, i)) \in \{ \star \} \times (M_A \times \mathbb{N}) \mid \star \vdash_A \hat{a} \ \! \} \\ \cup \{ ((a, i), (a', i)) \in (M_A \times \mathbb{N}) \times (M_A \times \mathbb{N}) \mid a \vdash_A a' \ \! \}$;

\item $P_{\oc A} \stackrel{\mathrm{df. }}{=} \{ \boldsymbol{s} \in \mathscr{L}_{\oc A} \mid \forall i \in \mathbb{N} . \ \! \boldsymbol{s} \upharpoonright i \in P_A \ \! \}$, where $\boldsymbol{s} \upharpoonright i$ is the j-subsequence of $\boldsymbol{s}$ that consists of moves $(a, i)$ yet changed into $a$;

\item $\boldsymbol{s} \simeq_{\oc A} \boldsymbol{t} \stackrel{\mathrm{df. }}{\Leftrightarrow} \exists \varphi \in \mathcal{P}(\mathbb{N}) . \ \! \forall i \in \mathbb{N} . \ \! \boldsymbol{s} \upharpoonright \varphi (i) \simeq_A \boldsymbol{t} \upharpoonright i \wedge \pi_2^\ast (\boldsymbol{s}) = (\varphi \circ \pi_2)^\ast(\boldsymbol{t})$, where $\mathcal{P}(\mathbb{N})$ is the set of all permutations of natural numbers.

\end{itemize}
\end{definition}

\begin{lemma}[Well-defined constructions on games]
\label{LemWellDefinedConstructionsOnGames}
Games (resp. wf-games) are closed under $\otimes$, $\multimap$, $\&$, $\oplus$ and $\oc$.
%Moreover, $\multimap$ and $\&$ preserve well-openness of games.  
\end{lemma}
\begin{proof}
Similarly to the corresponding proof in \cite{mccusker1998games}.
\end{proof}

Next, let us recall another central notion of \emph{strategies}:
\begin{definition}[Strategies \cite{mccusker1998games}]
\label{DefStrategies}
A \emph{\bfseries strategy} on a game $G$ is a non-empty subset $\sigma \subseteq P_G^{\mathsf{Even}}$, written $\sigma : G$, that satisfies:
\begin{itemize}

\item \textsc{(S1)} \emph{Even-prefix-closed} (i.e., $\forall \boldsymbol{s}mn\in \sigma . \ \! \boldsymbol{s} \in \sigma$);

\item \textsc{(S2)} \emph{Deterministic} (i.e., $\forall \boldsymbol{s}mn, \boldsymbol{s'}m'n' \in \sigma . \ \! \boldsymbol{s}m = \boldsymbol{s'}m' \Rightarrow \boldsymbol{s}mn = \boldsymbol{s'}m'n'$).

\end{itemize}
\end{definition}

As positions of a game $G$ are to be identified up to $\simeq_G$, we must identify strategies on $G$ up to $\simeq_G$, leading to: 
\begin{definition}[Identification of strategies \cite{mccusker1998games}]
\label{DefIdentificationOfStrategies}
The \emph{\bfseries identification of strategies} on a game $G$, written $\simeq_G$, is the relation on strategies $\sigma, \tau : G$ given by $\sigma \simeq_G \tau \stackrel{\mathrm{df. }}{\Leftrightarrow}  \forall \boldsymbol{s} \in \sigma, \boldsymbol{t} \in \tau . \ \! \boldsymbol{s} m \simeq_G \boldsymbol{t} l \Rightarrow \forall \boldsymbol{s} m n \in \sigma . \ \! \exists \boldsymbol{t} l r \in \tau . \ \! \boldsymbol{s} m n \simeq_G \boldsymbol{t} l r \wedge \forall \boldsymbol{t} l r \in \tau . \ \! \exists \boldsymbol{s} m n \in \sigma . \ \! \boldsymbol{t} l r \simeq_G \boldsymbol{s} m n$.
A strategy $\sigma : G$ is \emph{\bfseries valid} if $\sigma \simeq_G \sigma$.
\end{definition}

The identification $\simeq_G$ of strategies on each game $G$ forms a \emph{partial equivalence relation (PER)}; see \cite{mccusker1998games,abramsky2000full}.
 
%We are particularly concerned with strategies \emph{identified with themselves}:

Next, we need to focus on strategies that behave as \emph{proofs}, which we call \emph{winning} ones:
\begin{definition}[Winning of strategies]
A strategy $\sigma : G$ is:
\begin{itemize}

\item \emph{\bfseries Total} if $\bm{s} \in \sigma\wedge \bm{s} m \in P_G$ implies $\exists \bm{s} m n \in \sigma$ \cite{clairambault2010totality,abramsky1997semantics};

\item \emph{\bfseries Innocent} if $\bm{s}mn, \bm{t} \in \sigma \wedge \bm{t} m \in P_G \wedge \lceil \bm{t} m \rceil_G = \lceil \bm{s} m \rceil_G$ implies $\bm{t}mn \in \sigma \wedge \lceil \bm{t} m n \rceil_G = \lceil \bm{s} m n \rceil_G$ \cite{hyland2000full,mccusker1998games,abramsky1999game};

%\item \emph{\bfseries Well-bracketed (wb)} if, whenever $\bm{s} q \bm{t} a \in \sigma$, where $q$ is a question that justifies an answer $a$, every question in $\bm{\tilde{t}}$, where $\lceil \bm{s} q \bm{t} \rceil_G = \bm{\tilde{s}} q \bm{\tilde{t}}$, justifies an answer in $\bm{\tilde{t}}$ \cite{hyland2000full,abramsky1999game};

\item \emph{\bfseries Noetherian} if $\sigma$ does not contain any strictly increasing infinite sequence of P-views of positions of $G$ \cite{clairambault2010totality};

\item \emph{\bfseries Winning} if it is innocent, total and noetherian.

\end{itemize}
In addition, an innocent strategy $\sigma : G$ is \emph{\bfseries finite} if the set $\lceil \sigma \rceil_G \stackrel{\mathrm{df. }}{=} \{ \lceil \boldsymbol{s} \rceil_G \mid \boldsymbol{s} \in \sigma \ \! \}$ of all P-views of $\sigma$ is finite. 
\end{definition}

A conceptual explanation of winning is as follows.
First, a proof or an `argument' for the truth of a formula should not get `stuck', and thus, strategies for proofs must be total.
In addition, since logic is concerned with the truth of formulas, which are invariant w.r.t. `passage of time', proofs should not depended on \emph{states}; thus, it makes sense to impose \emph{innocence} on strategies for proofs \cite{hyland2000full,abramsky1999game}.
Next, recall that totality is not preserved under \emph{composition} of strategies \cite{abramsky1997semantics}, but it can be solved by additing noetherianity \cite{clairambault2010totality}.
It conceptually makes sense too because if a play by an innocent, noetherian strategy keeps growing \emph{infinitely}, then it cannot be Player's `intention', and therefore, it should result in \emph{win} for Player.
%In fact, however, it is win even in a stronger sense: Every play by an innocent, noetherian strategy is finite \cite{clairambault2010totality}.

In addition, let us introduce the game-semantic counterpart of \emph{linearity} of proofs in logic \cite{girard1987linear}:
\begin{definition}[Linearity of strategies]
\label{DefLinearityOfStrategies}
A j-sequence $\boldsymbol{s}$ is \emph{\bfseries linear}, written $\mathsf{L}(\boldsymbol{s})$, if, for each even-length prefix $\boldsymbol{t}$ of $\boldsymbol{s}$, an initial move $q$ in $\boldsymbol{t}$ justifies exactly one question $q'$ in $\boldsymbol{t}$, and the number of answers justified by $q'$ equals that of answers justified by $q$ in $\boldsymbol{t}$.
A strategy $\sigma : G$ is \emph{\bfseries linear} if $\forall \boldsymbol{s} \in \sigma . \ \! \mathsf{L}(\boldsymbol{s})$.
\end{definition}

\if0
\begin{definition}[Full exploitation of games]
\label{DefFullExploitation}
A position $\boldsymbol{s}$ of a LL-game $A$ \emph{\bfseries fully exploits} $A$, written $\mathsf{FE}(\boldsymbol{s}, A)$, if it satisfies:
\begin{itemize}

\item $A = \top$, and $\boldsymbol{s} = \boldsymbol{\epsilon}$; 

\item $A = \bot$, and $\boldsymbol{s} = q$; 

\item $\exists \clubsuit \in \{ \&, \oplus \} . \ \! A = A_1 \clubsuit A_2$, and $\exists i \in \overline{2} . \ \! \mathsf{FE}(\boldsymbol{s} \upharpoonright A_i, A_i)$;

\item $\exists \spadesuit \in \{ \&, \oplus \} . \ \! A = A_1 \spadesuit A_2$, and $\forall i \in \overline{2} . \ \! \mathsf{FE}(\boldsymbol{s} \upharpoonright A_i, A_i)$;

\item $\exists f \in \{ \oc, \wn \} . \ \! A = f A'$.

\end{itemize}
\end{definition}

\begin{definition}[Greed of strategies]
\label{DefGreedOfStrategies}
A strategy $\phi : A \multimap B$ is \emph{\bfseries greedy} if $\forall \boldsymbol{s} \in \phi . \ \! \mathsf{FE}(\boldsymbol{s} \upharpoonright B, B) \Rightarrow \mathsf{FE}(\boldsymbol{s} \upharpoonright A, A)$.
\end{definition}
\fi

Also, we slightly generalize \emph{strictness} of strategies in \cite{laurent2002polarized}:
\begin{definition}[Strictness of strategies]
A strategy $\phi : A \multimap B$ is \emph{\bfseries strict} if $\forall \boldsymbol{s} m n \in \sigma . \ \! m \in M_B^{\mathsf{Init}} \Rightarrow n \in M_A^{\mathsf{Init}}$.
\end{definition}

Next, let us proceed to recall standard constructions on strategies.
The simplest strategies are the following: 
\begin{definition}[Copy-cats \cite{abramsky1994games,abramsky2000full,hyland2000full,mccusker1998games}]
\label{DefCopyCats}
The \emph{\bfseries copy-cat (strategy)} $\mathit{cp}_A$ on a game $A$ is defined by:
\begin{equation*}
\mathit{cp}_A \stackrel{\mathrm{df. }}{=} \{\boldsymbol{s} \in P_{A_{[0]} \multimap A_{[1]}}^\mathsf{Even} \mid \forall \boldsymbol{t} \preceq{\boldsymbol{s}}. \ \mathsf{Even}(\boldsymbol{t}) \Rightarrow \boldsymbol{t} \upharpoonright 0 = \boldsymbol{t} \upharpoonright 1 \ \! \}
\end{equation*}
where the subscripts $(\_)_{[i]}$ on $A$ are to distinguish the two copies of $A$, and $\boldsymbol{t} \upharpoonright i \stackrel{\mathrm{df. }}{=} \boldsymbol{t} \upharpoonright A_{[i]}$ ($i = 0, 1$).
\end{definition}

%As the name indicates, the copy-cat $\mathit{cp}_A$ plays simply by `copy-catting' the last O-move.

\begin{lemma}[Well-defined copy-cats \cite{abramsky1994games,mccusker1998games}]
\label{LemWellDefinedCopyCats}
Given a game $A$, $\mathit{cp}_A$ is a valid, innocent, total, linear, strict strategy on the game $A \multimap A$. 
In addition, it is noetherian if $A$ is wf.
\end{lemma}
\begin{proof}
We just show that $\mathit{cp}_A$ is noetherian if $A$ is wf for the other points are trivial, e.g., validity of $\mathit{cp}_A$ is immediate from the definition of $\simeq_{A \multimap A}$. 
Given $\boldsymbol{s}mm \in \mathit{cp}_A$, it is easy to see by induction on $|\boldsymbol{s}|$ that  the P-view $\lceil \boldsymbol{s} m \rceil_{A \multimap A}$ is of the form $m_1 m_1 m_2 m_2 \dots m_k m_k m$, and thus, there is a sequence $\star \vdash_A m_1 \vdash_A m_2 \dots \vdash_A m_k \vdash_A m$. 
Therefore, if $A$ is wf, then $\mathit{cp}_A$ must be noetherian.
\end{proof}

Next, let us recall \emph{composition} and \emph{tensor} of strategies:
\begin{definition}[Composition of strategies \cite{mccusker1998games}]
Given games $A$, $B$ and $C$, and strategies $\phi : A \multimap B$ and $\psi : B \multimap C$, the \emph{\bfseries parallel composition} $\phi \parallel \psi$ of $\phi$ and $\psi$ is given by:
\begin{align*}
\phi \parallel \psi &\stackrel{\mathrm{df. }}{=} \{ \boldsymbol{s} \in \mathscr{J}_{((A \multimap B_{[0]}) \multimap B_{[1]}) \multimap C} \mid \boldsymbol{s} \upharpoonright A, B_{[0]} \in \phi, \\ &\boldsymbol{s} \upharpoonright B_{[1]}, C \in \psi, \boldsymbol{s} \upharpoonright B_{[0]}, B_{[1]} \in \mathit{pr}_{B} \ \! \}
\end{align*}
where the subscripts $(\_)_{[i]}$ on $B$ ($i = 0, 1$) are to distinguish the two copies of $B$, $\boldsymbol{s} \upharpoonright A, B_{[0]}$ (resp. $\boldsymbol{s} \upharpoonright B_{[1]}, C$, $\boldsymbol{s} \upharpoonright B_{[0]}, B_{[1]}$) is the j-subsequence of $\boldsymbol{s}$ that consists of moves of $A$ or $B_{[0]}$ (resp. $B_{[1]}$ or $C$, $B_{[0]}$ or $B_{[1]}$), and $\mathit{pr}_B \stackrel{\mathrm{df. }}{=} \{ \boldsymbol{s} \in P_{B_{[0]} \multimap B_{[1]}} \mid \forall \boldsymbol{t} \preceq{\boldsymbol{s}}. \ \mathsf{Even}(\boldsymbol{t}) \Rightarrow \boldsymbol{t} \upharpoonright 0 = \boldsymbol{t} \upharpoonright 1 \ \! \}$.

The \emph{\bfseries composition} $\phi ; \psi$ (or $\psi \circ \phi$) of $\phi$ and $\psi$ is defined by: 
\begin{equation*}
\phi ; \psi \stackrel{\mathrm{df. }}{=} \{ \boldsymbol{s} \upharpoonright A, C \mid \boldsymbol{s} \in \phi \! \parallel \! \psi \ \! \}
\end{equation*}
where $\boldsymbol{s} \upharpoonright A, C$ is the j-subsequence of $\boldsymbol{s}$ that consists of moves of $A$ or $B$.
\end{definition}

That is, the composition $\phi ; \psi : A \multimap C$ plays implicitly on $((A \multimap B_{[0]}) \multimap B_{[1]}) \multimap C$, employing $\phi$ if the last O-move is of $A$ or $B_{[0]}$, and $\psi$ otherwise, while Opponent plays on $A \multimap C$, where $\phi$ and $\psi$ communicate with each other via moves of $B_{[0]}$ or $B_{[1]}$, but it is `hidden' from Opponent.

\begin{lemma}[Well-defined composition of strategies \cite{mccusker1998games,clairambault2010totality}]
\label{LemWellDefinedCompositionOfStrategies}
Given games $A$, $B$ and $C$, and strategies $\phi : A \multimap B$ and $\psi : B \multimap C$, $\phi ; \psi$ is a strategy on the game $A \multimap C$. 
If $\phi$ and $\psi$ are winning (resp. linear, strict), then so is $\phi ; \psi$.
Given strategies $\phi' : A \multimap B$ and $\psi' : B \multimap C$ such that $\phi \simeq_{A \multimap B} \phi'$ and $\psi \simeq_{B \multimap C} \psi'$, we have $\phi ; \psi \simeq_{A \multimap C} \phi' ; \psi'$.
\end{lemma}
\begin{proof}
It is well-known that innocent strategies are closed under composition \cite{abramsky1999game,mccusker1998games}. 
Also, it is shown in \cite{clairambault2010totality} that the conjunction of innocence, totality and noetherianity is preserved under composition.
Finally, composition clearly preserves linearity, strictness and identification of strategies.
\end{proof}

%Next, let us recall \emph{tensor} $\otimes$ of strategies:
\begin{definition}[Tensor product of strategies \cite{abramsky1994games,mccusker1998games}]
\label{DefTensorOfStrategies}
Given games $A$, $B$, $C$ and $D$, and strategies $\phi : A \multimap C$ and $\psi : B \multimap D$, the \emph{\bfseries tensor (product)} $\phi \otimes \psi$ of $\phi$ and $\psi$ is given by:
\begin{equation*}
\phi \otimes \psi \stackrel{\mathrm{df. }}{=} \{ \boldsymbol{s} \in \mathscr{L}_{A \otimes B \multimap C \otimes D} \mid \boldsymbol{s} \upharpoonright A, C \in \phi, \boldsymbol{s} \upharpoonright B, D \in \psi \ \! \}
\end{equation*}
where $\boldsymbol{s} \upharpoonright A, C$ (resp. $\boldsymbol{s} \upharpoonright B, D$) is the j-subsequence of $\boldsymbol{s}$ that consists of moves of $A$ or $C$ (resp. $B$ or $D$).
\end{definition}

Intuitively the tensor $\phi \otimes \psi : A \otimes B \multimap C \otimes D$ plays by $\phi$ if the last O-move is of $A$ or $C$, and by $\psi$ otherwise.

\if0
\begin{lemma}[Well-defined tensor of strategies \cite{abramsky1994games,mccusker1998games}]
\label{LemWellDefinedTensorOfStrategies}
Given games $A$, $B$, $C$ and $D$, and strategies $\phi : A \multimap C$ and $\psi : B \multimap D$, $\phi \otimes \psi$ is a strategy on  $A \otimes B \multimap C \otimes D$. 
If $\phi$ and $\psi$ are winning (resp. linear), then so is $\phi \otimes \psi$.
Given strategies $\phi' : A \multimap C$ and $\psi' : B \multimap D$ such that $\phi \simeq_{A \multimap C} \phi'$ and $\psi \simeq_{B \multimap D} \psi'$, we have $\phi \otimes \psi \simeq_{A \otimes B \multimap C \otimes D} \phi' \otimes \psi'$.
\end{lemma}
\fi
\if0
\begin{proof}
Straightforward; see \cite{abramsky1994games,abramsky1999game,mccusker1998games,abramsky2000full}.
\end{proof}
\fi

Let us leave the details of \emph{\bfseries pairing} $\langle \_, \_ \rangle$, \emph{\bfseries copairing} $[\_, \_]$, \emph{\bfseries promotion} $(\_)^\dagger$ and \emph{\bfseries derelictions} $\mathit{der}$ to \cite{mccusker1998games} for lack of space.

\begin{lemma}[Well-defined constructions on strategies \cite{mccusker1998games}]
\label{LemWellDefinedConstructionsOnStrategies}
Given games $A$, $B$, $C$ and $D$, and strategies $\phi : A \multimap C$ and $\psi : B \multimap D$, $\phi \otimes \psi$ is a strategy on  $A \otimes B \multimap C \otimes D$. 
If $\phi$ and $\psi$ are winning (resp. linear, strict), then so is $\phi \otimes \psi$.
Given strategies $\phi' : A \multimap C$ and $\psi' : B \multimap D$ such that $\phi \simeq_{A \multimap C} \phi'$ and $\psi \simeq_{B \multimap D} \psi'$, we have $\phi \otimes \psi \simeq_{A \otimes B \multimap C \otimes D} \phi' \otimes \psi'$.
Similar statements hold for pairing and promotion. 
The dereliction $\mathit{der}_A$ is a valid, innocent, total, linear, strict strategy on $\oc A \multimap A$; in addition, it is noetherian if $A$ is wf.
\end{lemma}

%We are now ready to give:
\begin{definition}[Category $\mathcal{LG}$]
The category $\mathcal{LG}$ is given by:
\begin{itemize}

\item Objects are wf-games;

\item Morphisms $A \rightarrow B$ are the equivalence classes $[\phi] \stackrel{\mathrm{df. }}{=} \{ \phi' : A \multimap B \mid \phi \simeq_{A \multimap B} \phi' \ \! \}$ of valid, winning, linear strategies $\phi : A \multimap B$;

\item Composition of morphisms $[\phi] : A \rightarrow B$ and $[\psi] : B \rightarrow C$ is given by $[\psi] \circ [\phi] \stackrel{\mathrm{df. }}{=} [\psi \circ \phi] : A \rightarrow C$;

\item Identities are given by $\mathit{id}_A \stackrel{\mathrm{df. }}{=} [\mathit{cp}_A] : A \rightarrow A$.

\end{itemize}
\end{definition}

\begin{theorem}[NSC $\mathcal{LG}$]
\label{ThmLG}
The tuple $\mathcal{LG} = (\mathcal{LG}, \otimes, \top, \multimap, \oc)$ forms a NSC with finite products $(1, \&)$.
\end{theorem}
\begin{proof}
As outlined in \cite{hyland1997game} and by Lem.~\ref{LemWellDefinedConstructionsOnGames}, \ref{LemWellDefinedCopyCats} and \ref{LemWellDefinedConstructionsOnStrategies} (constructions on strategies are lifted to their equivalence classes, and games are wf for copy-cats to be noetherian). %and linearity of strategies is preserved under currying $\lambda_{A, B, C} : \mathcal{LG}(A \otimes B, C) \cong \mathcal{LG}(A, B \multimap C)$.
\end{proof}

\if0
Moreover, with the \emph{bang lemma} \cite{mccusker1998games}, it is straightforward to verify: 
\begin{corollary}[BwL of $\boldsymbol{\mathcal{LG}}$]
The NSC $\mathcal{LG}$ is a BwLSMCC.
\end{corollary}
\fi

\if0
Finally, let us recall the game-semantic counterpart of \emph{intuitionistic restriction} of proofs:
\begin{definition}[Well-bracketing of strategies \cite{hyland2000full,mccusker1998games}]
A strategy $\sigma : G$ is \emph{\bfseries well-bracketed (wb)} if given $\boldsymbol{s} q \boldsymbol{t} a \in \sigma$, where $\lambda_G^{\mathsf{QA}}(q) = \mathsf{Q}$, $\lambda_G^{\mathsf{QA}}(a) = \mathsf{A}$ and $\mathcal{J}_{\boldsymbol{s}q\boldsymbol{t}a}(a) = q$, each question in $\boldsymbol{t'}$, defined by $\lceil \boldsymbol{s} q \boldsymbol{t} \rceil_G = \lceil \boldsymbol{s} q \rceil_G . \boldsymbol{t'}$, justifies an answer in $\boldsymbol{t'}$.
\end{definition}
\fi

For lack of space, we leave the details of \emph{\bfseries well-bracketing (wb)} of strategies to \cite{hyland2000full,mccusker1998games}.
It is easy to show: %that atomic strategies in $\mathcal{LG}$ are all wb, and each construction on strategies in $\mathcal{LG}$ preserves wb.
\begin{corollary}[NSC $\mathcal{LG}^{\mathsf{wb}}$]
The lluf subcategory $\mathcal{LG}^\mathsf{wb}$ of $\mathcal{LG}$, in which for each morphism $[\phi]$ the strategy $\phi$ is wb, forms a subNSC of $\mathcal{LG}$ with finite products $(1, \&)$.
\end{corollary}

\subsection{Game-Semantic BiLSMCC}
Now, let us define a FwL-structure $(\invamp, \bot, \wn)$ on:
%Importantly, strictness is preserved under all the constructions on strategies in $\mathcal{LG}$ except currying, leading to:
\begin{definition}[Subcategory $\sharp\mathcal{LG}$]
The lluf subcategory $\sharp\mathcal{LG}$ of $\mathcal{LG}$ has exactly morphisms $[\phi]$ in $\mathcal{LG}$ such that $\phi$ is strict. 
\end{definition}

Clearly, $\sharp\mathcal{LG}$ is not closed, but Thm.~\ref{ThmLG} immediately gives:
\begin{lemma}[BwLSMC $\sharp \mathcal{LG}$]
\label{LemSharpLG}
The category $\sharp\mathcal{LG}$ together with the triple $(\otimes, \top, \oc)$ inherited from $\mathcal{LG}$ is a BwLSMC with finite coproducts $(0, \oplus)$ (n.b., they are weak in $\mathcal{LG}$ as in \cite{mccusker1998games}).
\end{lemma}

%First, let us introduce game-semantic \emph{par} $\invamp$ and \emph{why not} $\wn$:
\begin{definition}[Par on games]
\label{DefParOnGames}
The \emph{\bfseries par of games} $A$ and $B$ is the game $A \invamp B$ defined by:
\begin{itemize}

\item $M_{A \invamp B} \stackrel{\mathrm{df. }}{=} (M_A^{\mathsf{Init}} \times M_B^{\mathsf{Init}}) + M_A + M_B$;

\item $\lambda_{A \invamp B} : (\hat{a}, \hat{b}) \in M_A^{\mathsf{Init}} \times M_B^{\mathsf{Init}} \mapsto \mathsf{OQ}, a \in M_A \mapsto \lambda_A(a), b \in M_B \mapsto \lambda_B(b)$;

\item $\vdash_{A \invamp B} \stackrel{\mathrm{df. }}{=} \{ (\star, (\hat{a}, \hat{b})) \mid \star \vdash_A \hat{a}, \star \vdash_B \hat{b} \ \! \} + \vdash_A + \vdash_B \\ + \{ ((\hat{a}, \hat{b}), a) \in (M_A^{\mathsf{Init}} \times M_B^{\mathsf{Init}}) \times M_A \mid \hat{a} \vdash_A a \ \! \} \\ + \{ ((\hat{a}, \hat{b}), b) \in (M_A^{\mathsf{Init}} \times M_B^{\mathsf{Init}}) \times M_B \mid \hat{b} \vdash_B b \ \! \}$;

\item $P_{A \invamp B} \stackrel{\mathrm{df. }}{=} \{ \boldsymbol{s} \in \mathscr{L}_{A \invamp B} \mid \boldsymbol{s} \upharpoonright A \in P_A, \boldsymbol{s} \upharpoonright B \in P_B, \boldsymbol{s} = x \boldsymbol{t} \Rightarrow x \in M_A^{\mathsf{Init}} \times M_B^{\mathsf{Init}} \}$, where $\boldsymbol{s} \upharpoonright A$ (resp. $\boldsymbol{s} \upharpoonright B$) is the j-subsequence of $\boldsymbol{s}$ that consists of moves $(\hat{a}, \hat{b}) \in M_A^{\mathsf{Init}} \times M_B^{\mathsf{Init}}$ and $a \in M_A$ (resp. $b \in M_B$) with the former changed into $\hat{a}$ (resp. $\hat{b}$);

\item $\boldsymbol{s} \simeq_{A \invamp B} \boldsymbol{t} \stackrel{\mathrm{df. }}{\Leftrightarrow} \mathit{att}_{A \invamp B}^\ast(\boldsymbol{s}) = \mathit{att}_{A \invamp B}^\ast(\boldsymbol{t}) \wedge \boldsymbol{s} \upharpoonright A \simeq_A \boldsymbol{t} \upharpoonright A \wedge \boldsymbol{s} \upharpoonright B \simeq_B \boldsymbol{t} \upharpoonright B$, where $ \mathit{att}_{A \invamp B}$ is the function $M_{A \invamp B} \rightarrow \{ 0, 1, 2 \}$ that maps $(\hat{a}, \hat{b}) \in M_A^{\mathsf{Init}} \times M_B^{\mathsf{Init}} \mapsto 0,  a \in M_A \mapsto 1, b \in M_B \mapsto 2$.

\end{itemize}
\end{definition}

Dually to tensor $\otimes$, a position of $A \invamp B$ is an interleaving mixture of positions of $A$ and $B$ in which only Player may switch the $AB$-parity again by alternation.
Also, similarly to sum $\oplus$, only the first element of each position of $A \invamp B$ can be of the form $(\hat{a}, \hat{b}) \in M_A^{\mathsf{Init}} \times M_B^{\mathsf{Init}}$.
Note also that our par on games slightly generalizes that on \emph{wo}-games given in \cite{laurent2002polarized}. 

For instance, typical plays of $A \invamp B$ are as follows: 
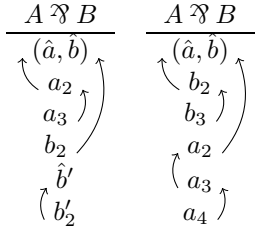
\begin{center}
\begin{tabular}{ccc}
$A \invamp B$ && $A \invamp B$ \\ \cline{1-1} \cline{3-3}
\tikzmark{ParC21} $(\hat{a}, \hat{b})$ \tikzmark{ParC23} && \tikzmark{ParC25} $(\hat{a}, \hat{b})$ \tikzmark{ParC27} \\
\tikzmark{ParD21} $a_2$ \tikzmark{ParC22} && \tikzmark{ParD25} $b_2$ \tikzmark{ParC26} \\
$a_3$ \tikzmark{ParD22} && $b_3$ \tikzmark{ParD26} \\
$b_2$ \tikzmark{ParD23} && \tikzmark{ParC28} $a_2$ \tikzmark{ParD27} \\
\tikzmark{ParC24} $\hat{b}'$ && \tikzmark{ParD28} $a_3$ \tikzmark{ParC29} \\
\tikzmark{ParD24} $b'_2$ && $a_4$ \tikzmark{ParD29}
\end{tabular}
\begin{tikzpicture}[overlay, remember picture, yshift=.25\baselineskip]
\draw [->] ({pic cs:ParD21}) [bend left] to ({pic cs:ParC21});
\draw [->] ({pic cs:ParD22}) [bend right] to ({pic cs:ParC22});
\draw [->] ({pic cs:ParD23}) [bend right] to ({pic cs:ParC23});
\draw [->] ({pic cs:ParD24}) [bend left] to ({pic cs:ParC24});
\draw [->] ({pic cs:ParD25}) [bend left] to ({pic cs:ParC25});
\draw [->] ({pic cs:ParD26}) [bend right] to ({pic cs:ParC26});
\draw [->] ({pic cs:ParD27}) [bend right] to ({pic cs:ParC27});
\draw [->] ({pic cs:ParD28}) [bend left] to ({pic cs:ParC28});
\draw [->] ({pic cs:ParD29}) [bend right] to ({pic cs:ParC29});
\end{tikzpicture}
\end{center}
where $\hat{a} a_2 a_3 a_4 \in P_A$, $\hat{b} b_2 \hat{b}' b'_2, b_1 b_2 b_3 \in P_B$, and the arrows represent the justification relation in the positions.

\begin{definition}[Par on strategies]
\label{DefParOnStrategies}
Given games $A$, $B$, $C$ and $D$, the \emph{\bfseries par of strategies} $\phi : A \multimap C$ and $\psi : B \multimap D$ is the subset $\phi \invamp \psi \subseteq P_{A \invamp B \multimap C \invamp D}^{\mathsf{Even}}$ defined by: 
\begin{equation*}
\phi \invamp \psi \stackrel{\mathrm{df. }}{=} \{ \boldsymbol{s} \in P_{A \invamp B \multimap C \invamp D}^{\mathsf{Even}} \mid  \boldsymbol{s} \upharpoonright A, C \in \phi, \boldsymbol{s} \upharpoonright B, D \in \psi \ \! \}
\end{equation*}
where $\boldsymbol{s} \upharpoonright A, C$ (resp. $\boldsymbol{s} \upharpoonright B, D$) is the j-subsequence of $\boldsymbol{s}$ that consists of moves $(\hat{a}, \hat{b}) \in M_A^{\mathsf{Init}} \times M_B^{\mathsf{Init}}$, $(\hat{c}, \hat{d}) \in M_C^{\mathsf{Init}} \times M_D^{\mathsf{Init}}$, $a \in M_A$ and $c \in M_C$ (resp. $b \in M_B$ and $d \in M_D$) with the first two changed into $\hat{a}$ and $\hat{c}$ (resp. $\hat{b}$ and $\hat{d}$), respectively.
\end{definition}

$\phi \invamp \psi$ may not satisfy the axiom S2 (Def.~\ref{DefStrategies}) unless $\phi$ and $\psi$ are both strict. 
If $\phi$ and $\psi$ are strict, $\phi \invamp \psi$ plays, e.g., as:
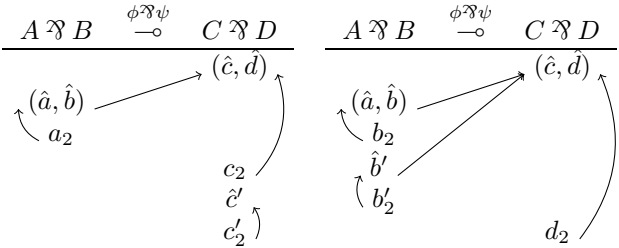
\begin{center}
\begin{tabular}{ccccccc}
$A \invamp B$ & $\stackrel{\phi \invamp \psi}{\multimap}$ & $C \invamp D$ && $A \invamp B$ & $\stackrel{\phi \invamp \psi}{\multimap}$ & $C \invamp D$ \\ \cline{1-3} \cline{5-7}
&&\tikzmark{ParC1} $(\hat{c}, \hat{d})$ \tikzmark{ParC3} &&&& \tikzmark{ParC5} $(\hat{c}, \hat{d})$ \tikzmark{ParC9} \\
\tikzmark{ParC2} $(\hat{a}, \hat{b})$ \tikzmark{ParD1}&& && \tikzmark{ParC6} $(\hat{a}, \hat{b})$ \tikzmark{ParD5} \\
\tikzmark{ParD2} $a_2$&& && \tikzmark{ParD6} $b_2$ \\
&& $c_2$ \tikzmark{ParD3} && \tikzmark{ParC8} $\hat{b}'$ \tikzmark{ParD7} \\
&& $\hat{c}'$ \tikzmark{ParC4} && \tikzmark{ParD8} $b'_2$ \\
&& $c'_2$ \tikzmark{ParD4} &&&& $d_2$ \tikzmark{ParD9}
\end{tabular}
\begin{tikzpicture}[overlay, remember picture, yshift=.25\baselineskip]
\draw [->] ({pic cs:ParD1}) to ({pic cs:ParC1});
\draw [->] ({pic cs:ParD2}) [bend left] to ({pic cs:ParC2});
\draw [->] ({pic cs:ParD3}) [bend right] to ({pic cs:ParC3});
\draw [->] ({pic cs:ParD4}) [bend right] to ({pic cs:ParC4});
\draw [->] ({pic cs:ParD5}) to ({pic cs:ParC5});
\draw [->] ({pic cs:ParD6}) [bend left] to ({pic cs:ParC6});
\draw [->] ({pic cs:ParD7}) to ({pic cs:ParC5});
\draw [->] ({pic cs:ParD8}) [bend left] to ({pic cs:ParC8});
\draw [->] ({pic cs:ParD9}) [bend right] to ({pic cs:ParC9});
\end{tikzpicture}
\end{center}
where $\hat{c} \hat{a} a_2 c_2 \hat{c}' c'_2 \in \phi$ and $\hat{d} \hat{b} b_2 \hat{b}' b'_2 d_2 \in \psi$.
Hence, $\invamp$ cannot be a bifunctor on $\mathcal{LG}$, but it can be on $\sharp\mathcal{LG}$:
\begin{definition}[Functor par]
The \emph{\bfseries functor par} is the bifunctor $\invamp$ on $\sharp\mathcal{LG}$ that maps objects $(A, B) \in \sharp\mathcal{LG} \times \sharp\mathcal{LG}$ to $A \invamp B \in \sharp\mathcal{LG}$, and morphisms $([\phi], [\psi]) \in \sharp\mathcal{LG} \times \sharp\mathcal{LG}((A, C), (B, D))$ to $[\phi \invamp \psi] \in \sharp\mathcal{LG}(A \invamp C, B \invamp D)$.
\end{definition}

\begin{lemma}[Well-defined par]
\label{LemWellDefinedPar}
The functor par $\invamp$ is indeed a well-defined bifunctor on $\sharp\mathcal{LG}$.
\end{lemma}
\begin{proof}
First, $\invamp$ on objects is clearly well-defined.
Next, $\invamp$ on strict strategies is well-defined, and it clearly preserves linearity, totality, noetherianity and identification of strategies.

For preservation of innocence, let $\phi : A \multimap B$ and $\psi : B \multimap D$ be innocent. 
Note that, during a play of the game $A \invamp B \multimap C \invamp D$, each O-move occurring in the codomain $C \invamp D$ cannot change the $CD$-parity, while the domain $A \invamp B$ part of each P-view must be that of $A$ or $B$.
Hence, the P-view of each element $\boldsymbol{s} \in \phi \invamp \psi$ is either the P-view of $\boldsymbol{s} \upharpoonright A, C \in \phi$ or $\boldsymbol{s} \upharpoonright B, D \in \psi$, whence $\phi \invamp \psi$ is innocent. 

Finally, $\invamp$ clearly preserves composition and identities. 
\end{proof}

\begin{definition}[Why not on games]
\label{DefWhyNotOnGames}
The \emph{\bfseries why not of a game} $A$ is the game $\wn A$ defined by:
\begin{itemize}
\item $M_{\wn A} \stackrel{\mathrm{df. }}{=} M_A^{\mathsf{Init}, \mathbb{N}} + (M_A \times \mathbb{N})$, where $M_A^{\mathsf{Init}, \mathbb{N}}$ is the set of all functions $\mathbb{N} \rightarrow M_A^{\mathsf{Init}}$; 

\item $\lambda_{\wn A} : \alpha \in M_A^{\mathsf{Init}, \mathbb{N}} \mapsto \mathsf{OQ}, (a, i) \in M_A \times \mathbb{N} \mapsto \lambda_A(a)$;

\item $\vdash_{\wn A} \stackrel{\mathrm{df. }}{=} (\{ \star \} \times M_A^{\mathsf{Init}, \mathbb{N}}) + (\{ \star \} \times (M_A^{\mathsf{Init}} \times \mathbb{N}))  \\ + \{ (\alpha, (a, i)) \in M_A^{\mathsf{Init}, \mathbb{N}} \times (M_A^{\mathsf{nInit}} \times \mathbb{N}) \mid \alpha(i) \vdash_A a \ \! \} \\ + \{ ((a, i), (a', i)) \in (M_A^{\mathsf{nInit}} \times \mathbb{N})^2 \mid a \vdash_A a' \ \! \}$; 

\item $P_{\wn A} \stackrel{\mathrm{df. }}{=} \{ \boldsymbol{s} \in \mathscr{L}_{\wn A} \mid \forall i \in \mathbb{N} . \ \! \boldsymbol{s} \upharpoonright i \in P_A, \boldsymbol{s} = x \boldsymbol{t} \Rightarrow x \in M_A^{\mathsf{Init}, \mathbb{N}} \}$, where $\boldsymbol{s} \upharpoonright i$ is the j-subsequence of $\boldsymbol{s}$ that consists of moves $\alpha \in M_A^{\mathsf{Init}, \mathbb{N}}$ and $(a, i) \in M_A \times \mathbb{N}$ yet changed into $\alpha(i)$ and $a$, respectively; 

\item $\boldsymbol{s} \simeq_{\wn A} \boldsymbol{t} \stackrel{\mathrm{df. }}{\Leftrightarrow} \exists \varphi \in \mathcal{P}(\mathbb{N}) . \ \! \forall i \in \mathbb{N} . \ \! \boldsymbol{s} \upharpoonright i \simeq_A \boldsymbol{t} \upharpoonright \varphi(i) \wedge (\varphi \circ \mathit{att}_{\wn A})^\ast(\boldsymbol{s}) = \mathit{att}_{\wn A}^\ast(\boldsymbol{t})$, where the function $\mathit{att}_{\wn A} : M_{\wn A} \rightarrow \{ \star \} + \mathbb{N}$ is given by $\alpha \mapsto \star$ and $(a, i) \mapsto i$. 
\end{itemize}
\end{definition}

\begin{definition}[Why not on strategies]
\label{DefWhyNotOnStrategies}
The \emph{\bfseries why not of a strategy} $\phi : A \multimap B$ is the subset $\wn\phi \subseteq P_{\wn A \multimap \wn B}^{\mathsf{Even}}$ given by:
\begin{equation*}
\wn \phi \stackrel{\mathrm{df. }}{=} \{ \boldsymbol{s} \in P_{\wn A \multimap \wn B}^{\mathsf{Even}} \mid \forall i \in \mathbb{N} . \ \! \boldsymbol{s} \upharpoonright i \in \phi \ \! \}
\end{equation*}
where $\boldsymbol{s} \upharpoonright i$ is the obvious analogue of that given in Def.~\ref{DefWhyNotOnGames}.
\end{definition}

Why not is essentially the infinite iteration of par, i.e., $\wn A \cong A \invamp A \invamp A \dots$ and $\wn \phi \cong \phi \invamp \phi \invamp \phi \dots$
A similar construction was introduced independently in \cite{harmer2007categorical} for a different purpose.
As outlined in the paper, we may lift $\wn$ to a monad on $\sharp\mathcal{LG}$:

\begin{definition}[Why not monad]
Given $A \in \sharp \mathcal{LG}$, strategies $\mathit{wst}_A : A \multimap \wn A$ and $\mathit{abs}_A : \wn \wn A \multimap \wn A$, called the \emph{\bfseries waste} and the \emph{\bfseries absorption} on $A$, respectively, are defined by:
\begin{align*}
\mathit{wst}_A &\stackrel{\mathrm{df. }}{=} \{ \boldsymbol{s} \in P_{A \multimap \wn A}^{\mathsf{Even}} \mid \forall \boldsymbol{t} \preceq \boldsymbol{s} . \ \! \mathsf{Even}(\boldsymbol{t}) \Rightarrow \boldsymbol{t} \upharpoonright A = \boldsymbol{t} \upharpoonright \wn A \upharpoonright 0 \ \! \} \\
\mathit{abs}_A &\stackrel{\mathrm{df. }}{=} \{ \boldsymbol{s} \in P_{\wn \wn A \multimap \wn A}^{\mathsf{Even}} \mid \forall \boldsymbol{t} \preceq \boldsymbol{s} . \ \! \mathsf{Even}(\boldsymbol{t}) \\ &\Rightarrow \forall i, j \in \mathbb{N} . \ \! \boldsymbol{t} \upharpoonright \wn \wn A \upharpoonright i \upharpoonright j = \boldsymbol{t} \upharpoonright \wn A \upharpoonright \langle i, j \rangle \ \! \}
\end{align*}
where $\langle \_, \_ \rangle$ is any fixed bijection $\mathbb{N} \times \mathbb{N} \stackrel{\sim}{\rightarrow} \mathbb{N}$.
The \emph{\bfseries why not monad} is the monad $\wn = (\wn, \eta, \mu)$ on $\sharp\mathcal{LG}$, where:
\begin{itemize}

\item The functor $\wn$ is given by $A \in \sharp \mathcal{LG} \mapsto \wn A \in \sharp \mathcal{LG}$, and $[\phi] \in \sharp \mathcal{LG}(A, B) \mapsto [\wn \phi] \in \sharp \mathcal{LG}(\wn A, \wn B)$;

\item The components of the natural transformations $\eta : \mathit{id}_{\sharp\mathcal{LG}} \Rightarrow \wn $ and $\mu : \wn \wn \Rightarrow \wn$ on each $A \in \sharp \mathcal{LG}$ are given by $\eta_A \stackrel{\mathrm{df. }}{=} [\mathit{wst}_A]$ and $\mu_A \stackrel{\mathrm{df. }}{=} [\mathit{abs}_A]$, respectively.

\end{itemize}
\end{definition}

%Essentially the same monad was introduced independently in \cite{harmer2007categorical} yet for a different purpose, viz., to analyze \emph{innocence} of strategies in terms of categories. 

\begin{lemma}[Well-defined why not]
\label{LemWellDefinedWhyNot}
The why not monad is a well-defined monad on $\sharp\mathcal{LG}$.
\end{lemma}
\begin{proof}
Similarly to the corresponding proof in \cite{harmer2007categorical}.
\end{proof}

Now, based on Lem.~\ref{LemSharpLG}, \ref{LemWellDefinedPar} and \ref{LemWellDefinedWhyNot}, it is easy to establish: 
\begin{theorem}[BiLSMCC $\mathcal{LG}$]
\label{ThmBiLSMCCLG}
The NSC $\mathcal{LG}$ together with the lluf subBwLSMC $\sharp\mathcal{LG}$, the triple $(\invamp, \bot, \wn)$, the obvious natural transformations $\Omega$, $\Sigma$ and $\Pi$, natural isomorphisms (\ref{NI1})-(\ref{NI4}), and the distributive law $\Upsilon : \oc \wn  \Rightarrow \wn \oc$ given by:
\begin{align*}
\Upsilon_A &\stackrel{\mathrm{df. }}{=} \{ \boldsymbol{s} \in P_{\oc \wn A \multimap \wn \oc A}^{\mathsf{Even}} \mid \forall \boldsymbol{t} \preceq \boldsymbol{s} . \ \! \mathsf{Even}(\boldsymbol{t}) \\ &\Rightarrow \forall i \in \mathbb{N} . \ \! \boldsymbol{t} \upharpoonright \oc \wn A \upharpoonright 0 \upharpoonright i = \pi_1^\ast(\boldsymbol{t} \upharpoonright \wn \oc A \upharpoonright i) \ \! \}
\end{align*}
where $\pi_1^\ast(\boldsymbol{t} \upharpoonright \wn \oc A \upharpoonright i)$ is obtained from $\boldsymbol{t} \upharpoonright \wn \oc A \upharpoonright i$ by replacing each occurrence $(a, j)$ with $a$, is a BiLSMCC.
\end{theorem}

\if0
\begin{theorem}[BiLSMCC of games and strategies]
The SMCC $\mathcal{LG}$ equipped with the triple $(\invamp, \bot, \wn)$ is BiL, and it satisfies:
\begin{enumerate}

\item $A \multimap B \cong \neg A \invamp B$, where $\neg A \stackrel{\mathrm{df. }}{=} A \multimap \bot$;

\item $1 \multimap B \cong B$

\item $A \invamp (B \& C) \cong (A \invamp B) \& (A \invamp C)$;

\item $\top \invamp A \cong \top$

\end{enumerate}
natural in $A, B, C \in \mathcal{LG}$.
\end{theorem}
\fi

\section{Cut-Elimination, Soundness and Completeness}
\label{CutElimSoundnessAndCompleteness}
Next, let us define cut-elimination processes on the calculi given in Sect.~\ref{SequentCalculi} by the game semantics given in Sect.~\ref{GameSemantics}. %and show (equational) soundness and completeness of the categorical semantics given in Sect.~\ref{CategoricalSemantics} w.r.t. the cut-eliminations. 

To define the cut-elimination processes, the following inductive, categorical notion plays a key role:
\begin{definition}[LL$^-$ morphisms]
\label{DefLinearityOfMorphisms}
Given a BiLSMCC $\mathcal{C}$, a morphism in $\sharp\mathcal{C}$ is \emph{\bfseries LL$\boldsymbol{^-}$} if it is of the following form:
\begin{equation}
\label{Linear}
A \stackrel{\varsigma}{\rightarrow} A' \stackrel{\phi}{\rightarrow} B' \stackrel{\varrho}{\rightarrow} B
\end{equation}
such that $\phi$, $\varsigma$ and $\varrho$ are morphisms in $\sharp\mathcal{C}$ inductively constructed respectively by the following grammars: 
\begin{itemize}

\item $\phi \stackrel{\mathrm{df. }}{=} \mathit{id} \mid \oc^\top \mid \oc^\bot \mid \phi \otimes \phi \mid \phi \invamp \phi \mid \langle \phi, \phi \rangle \mid [\phi, \phi] \mid \phi ; \phi \mid \oc \phi \mid \wn \phi \mid \lambda (\phi) \mid \lambda^{-1}(\phi)$

\item $\varsigma \stackrel{\mathrm{df. }}{=} \mathit{id} \mid \Omega \mid \Upsilon \mid \Pi \mid \alpha \mid \ell \mid \varpi \mid \epsilon \mid \delta \mid \theta \mid \pi_i \mid (\ref{NI2}) \mid \varsigma \otimes \varsigma \mid \varsigma \invamp \varsigma \mid \varsigma ; \varsigma$

\item $\varrho \stackrel{\mathrm{df. }}{=} \mathit{id} \mid \Omega \mid \Upsilon \mid \Sigma \mid \alpha \mid \ell \mid \varpi \mid \eta \mid \mu \mid \vartheta \mid \iota_i \mid (\ref{NI1}) \mid \varrho \otimes \varrho \mid \varrho \invamp \varrho \mid \varrho ; \varrho$

\end{itemize}
where $\oc^\top$ (resp. $\oc^\bot$) is the canonical one to $\top$ (resp. from $\bot$), $\alpha$, $\ell$ and $\varpi$ respectively range over associativities, units and symmetries w.r.t. $\otimes$ or $\invamp$, $\theta$ (resp. $\vartheta$) over natural isomorphisms of $\oc$ (resp. $\wn$), $\pi_i$ and $\iota_i$ over projections and injections ($i = 1, 2$), respectively, $\oc = (\oc, \epsilon, \delta)$, $\wn = (\wn, \eta, \mu)$, and $\lambda$ (resp. $\lambda^{-1}$) is currying w.r.t. the entire domain (resp. uncurrying w.r.t. the entire codomain).
\end{definition}

\begin{lemma}[Inductive semantics of LL$^-$]
\label{LemInductiveLinearity}
The interpretation of any proof in LLK$^-$ in any BiLSMCC is LL$^-$. 
\end{lemma}
\begin{proof}
By induction on proofs in LLK$^-$.
\end{proof}

\begin{lemma}[Inductive definability]
\label{LemInductiveDefinability}
Given a BiLSMCC $\mathcal{C}$, let $\Delta$ and $\Gamma$ be sequences of formulas of LL$^-$, and $f : \llbracket \Delta \rrbracket \rightarrow \llbracket \Gamma \rrbracket$ a LL$^-$-morphism in $\sharp\mathcal{C}$, where $\llbracket \Delta \rrbracket, \llbracket \Gamma \rrbracket \in \mathcal{C}$ are the interpretations of the sequences.
Then, there is a proof $p$ of a sequent $\Delta \vdash \Gamma$ in LLK$^-$ whose interpretation $\llbracket p \rrbracket$ equals $f$.
\end{lemma}
\begin{proof}
Since $f$ is LL$^-$, we may write $f = \varrho \circ \phi \circ \varsigma$; see (\ref{Linear}).
By the structural and the distribution rules in LLK$^-$ and naturality of $\varsigma$ and $\varrho$, $\varsigma$ and $\varrho$ may be excluded; thus, it suffices to prove definability of $\phi$.
Then, it is immediate by induction on $\phi$.
\end{proof}

By Lem.~\ref{LemInductiveLinearity} and \ref{LemInductiveDefinability}, we may first compute the interpretation $\llbracket p \rrbracket$ of any given proof $p$ in LLK$^-$ in $\sharp\mathcal{LG}$ (as defined in the proof of Thm.~\ref{ThmSemanticsOfLLNegative}) and then calculate the proof $\mathsf{nf}(p)$ in LLK$^-$, called the \emph{\bfseries normal-form} of $p$, from $\llbracket p \rrbracket$ such that $\llbracket \mathsf{nf}(p) \rrbracket = \llbracket p \rrbracket$ (as defined in the proof of Lem.~\ref{LemInductiveDefinability}).
Note that there is no Cut$^-$ occurring in $\mathsf{nf}(p)$, i.e., $\mathsf{nf}(p)$ is \emph{cut-free}; that is, we have defined a \emph{cut-elimination} process $\mathsf{nf}$ on LLK$^-$.
Combined with Thm.~\ref{ThmTranslationOfCLIntoLL}, it is not hard to give such a process on LK$^-$, and by the same method, on LLJ (without $\oplus$) and LJ (without $\vee$) as well, which for lack of space we omit.
To summarize:
\if0
\begin{definition}[Cut-elimination]
\label{DefCutElim}
Given a proof $p$ of a sequent $\Delta \vdash \Gamma$ in LLK$^-$ (resp. LK$^-$, LLJ, LJ), the \emph{\bfseries normal-form} $\mathsf{nf}(p)$ of $p$ is given by $\mathsf{nf}(p) \stackrel{\mathrm{df. }}{=} \mathcal{F}(\llbracket p \rrbracket)$, where $\llbracket \_ \rrbracket$ is the interpretation of the calculus in $\sharp\mathcal{LG}$ (resp. $\sharp\mathcal{LG}_\oc^\wn$, $\mathcal{LG}^{\mathsf{wc}}$, $\mathcal{LG}^{\mathsf{wc}}_\oc$), and $\mathcal{F}$ maps each $f : \llbracket \Delta \rrbracket \rightarrow \llbracket \Gamma \rrbracket$ to the proof $\mathcal{F}(f)$ of $\Delta \vdash \Gamma$ such that $\llbracket \mathcal{F}(f) \rrbracket = f$ (given in the proof of Thm.~\ref{ThmFullCompleteness}), where $\mathcal{LG}^{\mathsf{wc}}$ is the full subNSC of $\mathcal{LG}$ whose objects are wc.
\end{definition}
\fi

\begin{theorem}[Correctness]
Given a proof $p$ of a sequent $\Delta \vdash \Gamma$ in LLK$^-$ (resp. LK$^-$, LLJ, LJ), the normal-form $\mathsf{nf}(p)$ of $p$ is cut-free, and $\llbracket \mathsf{nf}(p) \rrbracket = \llbracket p \rrbracket$, where $\llbracket \_ \rrbracket$ is the interpretation of the calculus in $\sharp\mathcal{LG}$ (resp. $\sharp\mathcal{LG}_\oc^\wn$, $\mathcal{LG}^{\mathsf{wb}}$, $\mathcal{LG}^{\mathsf{wb}}_\oc$).
\end{theorem}

\begin{theorem}[Categorical soundness/completeness]
\label{ThmCategoricalSoundnessAndCompleteness}
Given a BiLSMCC $\mathcal{C}$, the interpretation of LLK$^-$ (resp. LK$^-$, LLJ, LJ) in $\sharp\mathcal{C}$ (resp. $\sharp\mathcal{C}_\oc^\wn$, $\mathcal{C}$, $\mathcal{C}_\oc$) is equationally sound and complete w.r.t. the cut-elimination defined above.
\end{theorem}
\begin{proof}
The soundness is by induction on proofs, and the completeness immediately follows from Thm.~\ref{ThmBiLSMCCLG}.
\end{proof}

\section{Full Completeness}
\label{FullCompleteness}
Lem.~\ref{LemInductiveDefinability} characterizes definable strategies only \emph{inductively}, which is not satisfactory as full completeness \emph{per se} (n.b., it was to define the cut-elimination procedure). 
This last main section addresses the problem: Focusing on finite, \emph{strongly linear} strategies, it gives a \emph{non-inductive} full completeness:

%Now, we may introduce a key condition of strategies that characterizes proof in LLK$^-$:
\begin{definition}[Strong linearity of strategies]
A strategy $\phi : A \multimap B$ is \emph{\bfseries strongly linear} if it is linear, strict, and satisfies: 
\begin{enumerate}

\item If $B = (X_1 \otimes X_2) \invamp Y$ for some games $X_1$, $X_2$ and $Y$, then $\phi = A \stackrel{\phi'}{\multimap} X_i \otimes (X_j \invamp Y) \stackrel{\Omega}{\multimap} (X_1 \otimes X_2) \invamp Y$ for some $\phi' : A \multimap X_i \otimes (X_j \invamp Y)$, where $i \neq j$;

\item If $A = X_1 \otimes (X_2 \invamp Y)$ for some games $X_1$, $X_2$ and $Y$, then $\phi = X_1 \otimes (X_2 \invamp Y) \stackrel{\Omega}{\multimap} (X_i \otimes X_j) \invamp Y \stackrel{\phi''}{\multimap} B$ for some $\phi'' : (X_i \otimes X_j) \invamp Y \multimap B$, where $i \neq j$;

\item If $B = \oc X \invamp \wn Y$ for some games $X$ and $Y$, then $\phi = A \stackrel{\phi'}{\multimap} \oc (A \invamp \wn Y) \stackrel{\Sigma}{\multimap} \oc X \invamp \wn Y$ for some $A \stackrel{\phi'}{\multimap} \oc (A \invamp \wn Y)$;

\item If $A = \oc X \otimes \wn Y$ for some games $X$ and $Y$, then $\phi = \oc X \otimes \wn Y \stackrel{\Pi}{\multimap} \wn (\oc A \invamp Y) \stackrel{\phi''}{\multimap} B$ for some $\wn (\oc A \invamp Y) \stackrel{\phi''}{\multimap} B$;

\item If $B = B_1 \oplus B_2$, then $\phi = A \stackrel{\phi_i}{\multimap} B_i \stackrel{\iota_i}{\multimap} B_1 \oplus B_2$ for some $i \in \overline{2}$ and $\phi_i : A \multimap B_i$.

\end{enumerate}
\end{definition}

\if0
\begin{lemma}[Distribution lemma]
Let $\phi : A \multimap B$ be a strict, linear, innocent strategy.
Then, we have the following: 
\begin{enumerate}

\item \label{FirstD} If $B = (X_1 \otimes X_2) \invamp Y$ for L-games $X_1$, $X_2$ and $Y$, then $\phi = A \stackrel{\phi'}{\multimap} X_i \otimes (X_j \invamp Y) \stackrel{\Omega}{\multimap} (X_1 \otimes X_2) \invamp Y$ for some strict, linear, innocent strategy $\phi'$, where $i \neq j$;

\item \label{SecondD} If $A = X_1 \otimes (X_2 \invamp Y)$ for L-games $X_1$, $X_2$ and $Y$, then $\phi = X_1 \otimes (X_2 \invamp Y) \stackrel{\Omega}{\multimap} (X_i \otimes X_j) \invamp Y \stackrel{\phi''}{\multimap} B$ for some strict, linear, innocent strategy $\phi''$, where $i \neq j$.

\end{enumerate}
\end{lemma}
\fi

\if0
\begin{lemma}[Bang lemma \cite{mccusker1998games}]
\label{LemBangLemma}
Given an innocent strategy $\phi : \oc A \multimap \oc B$, we have $(\mathit{der}_B \circ \phi)^\dagger \simeq_{\oc A \multimap \oc B} \phi$.
\end{lemma}

\begin{lemma}[Cobang lemma]
\label{LemCoBangLemma}
Given an innocent strategy $\psi : \wn A \multimap \wn B$, we have $\mathit{abs}_A \circ \wn (\phi \circ \mathit{wst}_A) \simeq_{\wn A \multimap \wn B} \psi$.
\end{lemma}
\begin{proof}
Symmetric to the proof of Lem.~\ref{LemBangLemma}.
\end{proof}
\fi 

\begin{theorem}[Full completeness]
\label{ThmFullCompleteness}
Let $\llbracket \Delta \rrbracket \rightarrow \llbracket \Gamma \rrbracket$ be the interpretation of a sequent $\Delta \vdash \Gamma$ without atoms in LLK$^-$ (resp. LLJ without $\oplus$) in $\sharp\mathcal{LG}$ (resp. $\mathcal{LG}^{\mathsf{wb}}$), and $[\phi] : \llbracket \Delta \rrbracket \rightarrow \llbracket \Gamma \rrbracket$ in the category such that $\phi$ is finite and strongly linear.
Then, there is a proof $p$ of $\Delta \vdash \Gamma$ in the calculus such that $\llbracket p \rrbracket = [\phi]$. 
\end{theorem}
\begin{proof}
%Let us dispense with the semantic bracket $\llbracket \_ \rrbracket$.
By finiteness and strong linearity of $\phi$, we may show full completeness of the interpretation of LLK$^-$ in $\sharp\mathcal{LG}$ by induction on $\Delta, \Gamma$.
Finally, full completeness of the interpretation of LLJ (without $\oplus$) in $\mathcal{LG}^{\mathsf{wb}}$ is shown just similarly. 
\end{proof}

\section{Conclusion and Future Work}
\label{ConclusionAndFutureWork}
We have given a unity of logic in terms of sequent calculi, categories and games.
As future work, we would like to develop \emph{term calculi} that match our semantics. 
We are also interested in extending the present work to \emph{predicate logic}.

%\subsubsection*{Acknowledgments.} The author acknowledges the financial support from Funai Overseas Scholarship. 

%\bibliographystyle{plain}
%\bibliography{CategoricalLogic,GamesAndStrategies,RecursionTheory,PCF,LogicalCalculus,TypeTheoriesAndProgrammingLanguages,HoTT,GoI,LinearLogic}      

%\section*{Acknowledgment}
%The author was supported by Funai Overseas Scholarship.

%\section{The References Section}
%\label{references}
\bibliographystyle{IEEEtran}
\bibliography{RecursionTheory,LogicalCalculus,HoTT,GoI,GamesAndStrategies,LinearLogic,PCF,TypeTheoriesAndProgrammingLanguages,CategoricalLogic,Recursion}

\end{document}